\documentclass{article}
\usepackage{amsmath, amssymb, amsthm}
\usepackage{graphicx}
\usepackage{epstopdf}
\usepackage{caption}
\usepackage{subcaption}
\usepackage{color}
\usepackage{fullpage}
\usepackage{tabularx}
\usepackage{enumerate}
\usepackage[vcentermath]{youngtab}
\usepackage[utf8]{inputenc}
\usepackage{young}
\allowdisplaybreaks

\newcommand{\Bea}{\begin{eqnarray*}}
\newcommand{\Eea}{\end{eqnarray*}}
\newcommand{\bea}{\begin{eqnarray}}
\newcommand{\eea}{\end{eqnarray}}
\newcommand{\f}{\frac}
\newcommand{\bs}{\backslash}

\newcommand{\C}{\mathbb{C}}
\newcommand{\E}{\mathbb{E}}
\newcommand{\N}{\mathbb{N}}
\newcommand{\R}{\mathbb{R}}

\newcommand{\CM}{\mathcal{M}}
\newcommand{\CN}{\mathcal{N}}

\newcommand{\comment}[1]{}

\def\e{\epsilon}
\def\g{\gamma}

\def\s{\sigma}
\def\l{\lambda}

\def\1{\textbf{1}}

\def\bs{\backslash}
\newtheorem{thm}{Theorem}
\newtheorem{prop}{Proposition}
\newtheorem{lem}{Lemma}

\begin{document}
\bibliographystyle{abbrv}

\title{Fluctuations of Linear Eigenvalue Statistics of Random Band Matrices}

\author{
\hspace{.25in} \parbox[t]
{0.3\textwidth}
{{\sc Indrajit Jana}
\thanks{ Department of Mathematics,
        University of California, Davis,
        CA 95616-8633. ijana@math.ucdavis.edu} }
\parbox[t]{0.3\textwidth}
{{\sc Koushik Saha}
\thanks{Department of Mathematics,
        Indian Institute of Technology,
         Bombay, Mumbai 400076, India.  ksaha@math.iitb.ac.in, Research
supported by INSPIRE Fellowship, Dept. of Science and Technology, Govt. of India and VIMSS fellowship.}}
\parbox[t]{0.4\textwidth}{{\sc Alexander Soshnikov}
\thanks{ Department of Mathematics,
        University of California, Davis,
        CA 95616-8633, soshniko@math.ucdavis.edu.}} }

\date{\today}
\maketitle
\begin{abstract}
In this paper, we study the fluctuation of linear eigenvalue statistics of Random Band Matrices defined by $M_{n}=\f{1}{\sqrt{b_{n}}}W_{n}$, where
$W_{n}$ is a $n\times n$ band Hermitian random matrix of bandwidth $b_{n}$, i.e., the diagonal elements and only first $b_{n}$ off diagonal elements
are nonzero. We study the linear eigenvalue statistics $\CN(\phi)=\sum_{i=1}^{n}\phi(\l_{i})$ of such matrices, where $\l_{i}$ are the eigenvalues of
$M_{n}$ and $\phi$ is a sufficiently smooth function. We prove that $\sqrt{\f{b_{n}}{n}}\left[\CN(\phi)-\E \CN(\phi)\right]\stackrel{d}{\to} N(0,V(\phi))$ for $b_{n}>>\sqrt{n}$, where $V(\phi)$ is given in the Theorem 1.
\end{abstract}

\noindent\textbf{Keywords:} Band random matrix, Central limit theorem, Gaussian distribution, linear  eigenvalue statistics,  semi circular law, Wigner matrix.



\section{Introduction:}

Random Matrix Theory was developed from several different sources in the early 20th century. It is used as an important mathematical tool in various fields namely, Mathematics, Physics, wireless communication engineering etc. One of the earliest example of a random matrix appeared in
the study of sample covariance estimation was done by John Wishart ~\cite{wishart1928generalised}. In the early 1950s, Wigner introduced random matrix ensemble to study the energy spectra of heavy atoms undergoing slow nuclear reactions.


Random matrices are also used to model wireless channels. A random matrix model of CDMA networks can be found in ~\cite{tulino2004random,verdu1999spectral}.

A special kind of random matrix ensemble is a random band matrix. In 1955, Wigner studied the matrices $H$ of the form $H=K+V$, where $K$ is an $n\times n$ diagonal matrix consisting of $\cdots -2, -1, 0, 1, 2, \cdots$, and $V$ is an $n\times n$ symmetric sign matrix having non vanishing elements only up to a distance $b_{n}$
from the main diagonal. Such a matrix $H$ was called as bordered matrix  ~\cite{wigner1955,wigner1957}.


Another treatment of random band matrix was done by G. Casati et al. ~\cite{casati1990scaling,casati1991scaling} in the context of
Quantum Chaos. They studied $n\times n$ symmetric random band matrices of bandwidth $b_{n}$, where $b_{n}$ grows with $n$. In 1992, Molchanov et al. proved the Semicircle Law for random band matrices \cite{molchanov1992limiting}. In 1991, Fyodorov and Mirlin proved that $\f{b_{n}^{2}}{n}$ is a crucial parameter for random band matrices \cite{fyodorov1991scaling,mirlin1996transition}. Numerical simulations
show that the local eigenvalue statistics changes from Poisson to GOE or GUE as $b_{n}$ changes from $b_{n}<<\sqrt{n}$ to $b_{n}>>\sqrt{n}$.
Recently,  Li and Soshnikov \cite{li2013central} proved the Central Limit Theorem (CLT) for linear statistics of eigenvalues of band random matrices when
the bandwidth $b_{n}$ satisfies $\sqrt{n}<<b_{n}<<n$. In this article we write $\alpha_{n}<<\beta_{n}$ if $\f{\alpha_{n}}{\beta_{n}}\to 0$ as $n\to\infty$, and we write $\alpha_{n}=O(\beta_{n})$ if $\left|\f{\alpha_{n}}{\beta_{n}}\right|\leq C$ for all $n$ for some constant $C>0$.

In this article, we  deal with the CLT for the eigenvalue statistics of band random matrices. We take the approach of M. Shcherbina in \cite{shcherbina2011central} to establish the CLT for band matrices with bandwidth $b_n$ where $b_n\to \infty$ as $n\to \infty$. We give an alternative proof of Li and Soshnikov \cite{li2013central} result on CLT of band matrices when $\sqrt{n}<<b_{n}<<n$. We have given some simulation results in Section \ref{section: Simulation}, which ensure that the CLT for band matrices will also hold if $\sqrt n/b_n  \nrightarrow 0$ and $b_n\to \infty$.



Now we define our model. Let us define the (circular) distance function $d_{n}:\N\times \N\to\N$ as
\Bea
d_{n}(j,k):=\min\{|j-k|,n-|j-k|\},
\Eea
and the index sets $I_{n}, I_{n}^{+}\subset\N\times\N$, $I_{1}\subset\N$ as
\bea\label{defn: definition of I_1, I_n}
I_{n}:=\{(j,k):d_{n}(j,k)\leq b_{n}\},\;\;\;
I_{n}^{+}=\{(j,k):(j,k)\in I_{n},j\leq k\},\;\;\;
I_1=\{1< j \leq n: (1,j)\in I_n\}
\eea
where $\{b_{n}\}$ is a sequence of positive integers such that $b_{n}\to\infty$ as $n\to\infty$.

Define a real symmetric random band matrix $M=(m_{jk})_{n\times n}$ of bandwidth $b_{n}$ as
\bea\label{defn: random band matrix}
m_{jk}=m_{kj}=\left\{\begin{array}{ll}
b_{n}^{-1/2}w_{jk} & \text{if $d_{n}(j,k)\leq b_{n}$}\\
0 & \text{otherwise},
\end{array}\right.
\eea
where $\{w_{ii}\}$ and $\{w_{jk}\}_{j\neq k,(j,k)\in I_{n}^{+}}$ are two sets of iid real random variables with

\bea\label{eqn: moment condition}
\E[w_{jk}]=0,\;\;\;
\E[w_{jk}^{2}]=
\left\{\begin{array}{cl}
1 & \text{if $j\neq k$} \\
\s^{2} & \text{if $j=k$}.
\end{array}\right.
\eea
Here $\{w_{jk}\}$ may depend on $n$, but we suppress it when there is no confusion. Let $\l_{1}\leq\l_{2}\leq \cdots\leq \l_{n}$ be the eigenvalues of the random band matrix $M$. Define the linear eigenvalue statistic of the eigenvalues of $M$ as
\bea\label{eqn: linear eigenvalue statistic}
\CN_{n}(\phi)=\sum_{i=1}^{n}\phi(\l_{i}),
\eea
and  the normalized eigenvalue statistic of the matrix $M$ as
\bea\label{eqn: normalized linear eigenvalue statistic}
\CM_{n}(\phi)=\sqrt{\f{b_{n}}{n}}\CN_{n}(\phi),
\eea
where $\phi$ is a test function. 

\section{Main Results:}
\begin{thm}\label{thm: Main theorem 1}
 Let $M$ be a real symmetric random band matrix as defined in \eqref{defn: random band matrix}, and $b_{n}$ be a sequence of integers satisfying $\sqrt{n}<<b_{n}<<n$. Assume the following:
 \begin{enumerate}[(i)]
 \item $w_{jk}$ satisfies the Poincar\'e inequality with constant $m>0$ not depending on $j,k,n$ i.e., for any continuously differentiable function $f$,
 
 \Bea
 \text{Var}(f(w_{jk}))\leq \f{1}{m}\E\left[\left|f'(w_{jk})\right|^{2}\right].
 \Eea
  \item $\E[w_{jk}^{4}]=\mu_{4}$ for all $j\neq k$ and $d_{n}(j,k)\leq b_{n}$.
  \item $\phi:\R\to\R$ be a test function in the Sobolev space $H^{s}$ i.e., $\|\phi\|_{s}<\infty$, where
  \Bea
  \|\phi\|_{s}^{2}&=&\int_{\R}(1+2|t|)^{2s}|\hat \phi(t)|^{2}\;dt,\\
  \hat\phi(t)&=&\f{1}{\sqrt{2\pi}}\int_{\R}e^{-it\l}\phi(\l)\;d\l,
  \Eea
 and $s>5/2$.
 \end{enumerate}
Then the centred normalized eigenvalue statistic $\CM^{\circ}(\phi)=\CM_{n}(\phi)-\E[\CM_{n}(\phi)]$ converges in distribution to the Gaussian random variable with mean zero and variance given by
\Bea
V(\phi)&=&\f{\kappa_{4}}{16\pi^{2}}\left(\int_{-2\sqrt{2}}^{2\sqrt{2}}\f{4-\mu^{2}}{\sqrt{8-\mu^{2}}}\phi(\mu)\;d\mu\right)^{2}+\f{\s^{2}}{16\pi^{2}}\left(\int_{-2\sqrt{2}}^{2\sqrt{2}}\f{\mu\phi(\mu)}{\sqrt{8-\mu^{2}}}\;d\mu\right)^{2}\\
&&+\int_{-2\sqrt{2}}^{2\sqrt{2}}\int_{-2\sqrt{2}}^{2\sqrt{2}}\sqrt{(8-x^{2})(8-y^{2})} F(x,y) \int_{-2\sqrt{2}}^{2\sqrt{2}}\int_{-2\sqrt{2}}^{2\sqrt{2}}\f{\mu_{1}\phi(\mu_{1})}{(x-\mu_{1})\sqrt{8-\mu_{1}^{2}}}\f{\mu_{2}\phi(\mu_{2})}{(x-\mu_{2})^2\sqrt{8-\mu_{2}^{2}}}d\mu_{1}d\mu_{2}\ dx dy,
\Eea
where for $x\neq y$
\Bea
F(x,y)=2\int_{-\infty}^{\infty}\f{(s^{3}\sin s-s\sin^{3}s)}{2(s^{2}-\sin^{2}s)^{2}-(s^{3}\sin s+s\sin^{3}s)xy+s^{2}\sin^{2}s(x^{2}+y^{2})}ds,
\Eea
and $\kappa_{4}$ is the fourth cumulant of the off-diagonal entries, i.e., $\kappa_{4}=\mu_{4}-3$.
\end{thm}

\section{Proof of Theorem \ref{thm: Main theorem 1}:}
We will follow the approach taken by M. Shcherbina in \cite{shcherbina2011central} for full (Wigner) matrix. This approach is based on two main ideas. The first ingredient is stated in the following proposition which gives a bound on the variance of linear eigenvalue statistics with a sufficiently smooth test function in term of the variance of the trace of the resolvent of a random matrix. For a proof of this result see \cite{shcherbina2011central,soshnikov2013fluctuations}. In what follows, we denote $X^{\circ}=X-\E[X]$ for any random variable $X$.
\begin{prop}\label{prop: variance of N_n}
Let $M$ be an $n\times n$ real symmetric random matrix and $\mathcal N_n(\phi)$ be a linear eigenvalue statistic of its eigenvalue as in  \eqref{eqn: linear eigenvalue statistic}. Then for any $s>0$ 
we have
\Bea
\text{Var}[\CN_{n}(\phi)]\leq C_{s}\|\phi\|_{s}^{2}\int_{0}^{\infty}dy\;e^{-y}y^{2s-1}\int_{-\infty}^{\infty}\text{Var}[\text{Tr}(G(x+iy))]\;dx,
\Eea
where $C_{s}$ is a constant depends only on $s$, and $G(z)=(M-zI)^{-1}$, is the resolvent of the matrix $M$.
\end{prop}
The second ingredient of this approach is to use the martingale difference technique to provide a good bound on $\text{Var}(\gamma_n)$ where $\gamma_n$ is the trace of the resolvent of a matrix. The following proposition gives that bound.
\begin{prop}\label{prop: bound on Var(gamma)}
Consider symmetric band matrix $M$ defined in  \eqref{defn: random band matrix} and assume \eqref{eqn: moment condition} is satisfied. Then for some $C>0$ not depending on $z,n$ we have
\bea\label{eqn: bound on the variance of gamma_n}
\text{Var}\{\g_{n}\}\leq\f{Cn}{b_{n}}\left(y^{-2}+y^{-4}\right)\left(\max\left\{y,|x|-\f{2}{y}\right\}\right)^{-2}
\eea
where $ \g_{n}=\text{Tr}(M-zI)^{-1}=\text{Tr}(G)$ and $z=x+iy$, $y>0$.
\end{prop}
We prove this result in the appendix section. Now we  outline the proof of Theorem \ref{thm: Main theorem 1}
\begin{proof}[Proof of Theorem \ref{thm: Main theorem 1}:]
By L\'evy's continuity theorem, it suffices to show that if
\bea\label{defn: characteristic function}
Z_{n}(x)=\E[e_{n}(x)],\;\;\;e_{n}(x)=e^{ix\CM_{n}^{\circ}(\phi)}
\eea
then  for each $x\in\R$
\Bea
\lim_{n\to\infty}Z_{n}(x)=\exp\left[-\f{x^{2}V(\phi)}{2}\right],
\Eea
where $V(\phi)$ as in Theorem \ref{thm: Main theorem 1}.
For any test function $\phi\in H^{s}$, define
\Bea
\phi_{\eta}=P_{\eta}*\phi,
\Eea
where $P_{\eta}$ is the Poisson kernel given by
\Bea
P_{\eta}(x)=\f{\eta}{\pi(x^{2}+\eta^{2})}.
\Eea
We know that $\phi_{\eta}$ approximates $\phi$ in the $H^{s}$ norm i.e.,
\bea\label{eqn: approximation of phi}
\lim_{\eta\to 0}\|\phi-\phi_{\eta}\|_{s}\to 0.
\eea
For the moment, we denote the characteristic function defined in \eqref{defn: characteristic function}, by $Z_{n}(\phi)$ (to make its dependence on $\phi$ clear). Then we have 
\Bea
\lim_{n\to\infty}Z_{n}(\phi)=\lim_{\eta\downarrow 0}\lim_{n\to\infty}\left(Z_{n}(\phi)-Z_{n}(\phi_{\eta})\right)+\lim_{\eta\downarrow 0}\lim_{n\to\infty} Z_{n}(\phi_{\eta}).
\Eea
Now using the Proposition \ref{prop: variance of N_n} and \eqref{eqn: approximation of phi}, we shall show that
\bea\label{eqn: limit of Z(phi)-Z(phieta)}
\lim_{\eta\downarrow 0}\lim_{n\to\infty}\left(Z_{n}(\phi)-Z_{n}(\phi_{\eta})\right)=0.
\eea
and then
\Bea
\lim_{n\to\infty}Z_{n}(\phi)=\lim_{\eta\downarrow 0}\lim_{n\to\infty}Z_{n}(\phi_{\eta}).
\Eea
Hence it suffices to find the limit of
\bea\label{defn: characteristic function corresponding to phieta}
Z_{\eta,n}:=Z_{n}(\phi_{\eta})=\E\left[e_{\eta,n}(x)\right]
\eea
with
\Bea
e_{\eta,n}(x)=\exp\left[ix\CM_{n}^{\circ}(\phi_{\eta})\right]
\Eea
as $n\to\infty$ and $\eta\downarrow 0$ uniformly in $n$. Proofs of \eqref{eqn: limit of Z(phi)-Z(phieta)} and \eqref{defn: characteristic function corresponding to phieta} are given in the next two subsections and that will complete the proof of this theorem.
\end{proof}
\subsection{Proof of equation \eqref{eqn: limit of Z(phi)-Z(phieta)}:}
First observe that
\bea\label{eqn: Z_n(phi)-Z_n(phi_eta)}
\left|Z_{n}(\phi)-Z_{n}(\phi_{\eta})\right|^{2}\leq 2|x|^{2}\text{Var}\left[\CM_{n}(\phi)-\CM_{n}(\phi_{\eta})\right]
\leq 2|x|^{2}\f{b_{n}}{n}\text{Var}\left[\CN_{n}(\phi)-\CN_{n}(\phi_{\eta})\right].
\eea
Now, in view of Proposition \ref{prop: variance of N_n}, to bound $\text{Var}\left[\CN_{n}(\phi)-\CN_{n}(\phi_{\eta})\right]$ we need to estimate
\Bea
\int_{-\infty}^{\infty}\text{Var}\left(\g_{n}(x+iy)\right)\;dx,
\Eea
where $\g_{n}(x+iy)=\text{Tr}(G(x+iy))$ and $G(z)=(M-zI)^{-1}$.
We estimate that for $y>0$
\Bea
\int_{-\infty}^{\infty}\left(\max\left\{y,|x|-\f{2}{y}\right\}\right)^{-2}\;dx&\leq &\int_{||x|-2/y|<y}\f{1}{y^{2}}\;dx+\int_{||x|-2/y|\geq y}(x-2/y)^{-2}\;dx\\
&\leq&\f{10}{y}+10y
\Eea
Using the above estimate and \eqref{eqn: bound on the variance of gamma_n}, we have
\bea
\int_{0}^{\infty}dy\;e^{-y}y^{2s-1}\int_{-\infty}^{\infty}\text{Var}(\g_{n})\;dx&\leq&\f{C'}{b_{n}}\int_{0}^{\infty}e^{-y}y^{2s-1}4n\left(\f{1}{y}+y\right)\left(\f{1}{y^{2}}+\f{1}{y^{4}}\right)\;dy\nonumber\\
&=&C\f{n}{b_{n}}\int_{0}^{\infty}e^{-y}\left(2y^{2s-3-1}+y^{2s-1-1}+y^{2s-5-1}\right)\;dy\nonumber\\
&=&C\f{n}{b_{n}}\left(\Gamma(2s-3)+\Gamma(2s-1)+\Gamma(2s-5)\right).\label{eqn: estimate of rhs of proposition 1}
\eea
If we take
\Bea
s=\f{5}{2}+\e,\;\;\;\;\e>0
\Eea
then $\Gamma(2s-3)=\Gamma(2+2\e)$, $\Gamma(2s-1)=\Gamma(4+2\e)$, and $\Gamma(2s-5)=\Gamma(2\e)$. By Proposition \ref{prop: variance of N_n}, and \eqref{eqn: estimate of rhs of proposition 1}, we have
\Bea
\text{Var}\left(\CN_{n}(\phi)-\CN_n(\phi_\eta)\right)\leq C(\e)\f{n}{b_{n}}\|\phi-\phi_{\eta}\|_{s}.
\Eea
Using the above estimate and \eqref{eqn: Z_n(phi)-Z_n(phi_eta)}, we have
\Bea
\left|Z_{n}(\phi)-Z_{n}(\phi_{\eta})\right|^{2}
&\leq&2|x|^{2}\f{b_{n}}{n}\cdot C(\e)\f{n}{b_{n}}\|\phi-\phi_{\eta}\|_{s}\\
&=&2C(\e)|x|^{2}\|\phi-\phi_{\eta}\|_{s}\\
&\rightarrow& 0\;\;\;\text{as $\eta\to 0$}.
\Eea
The last limit follows from the equation \eqref{eqn: approximation of phi}. This completes the proof of \eqref{eqn: limit of Z(phi)-Z(phieta)}.

\subsection{Finding the limit of the characteristic function \eqref{defn: characteristic function corresponding to phieta}:}
We will be using the Lemma \ref{main lemma without poincare condition} and Lemma \ref{main lemma with poincare condition} from appendix in the proof of \eqref{defn: characteristic function corresponding to phieta}. Let us denote the averaging with respect to  $\{w_{1i}; 1\leq i\leq n\}$ by $\E_1$.

\begin{proof}[Proof of \eqref{defn: characteristic function corresponding to phieta}:]
Using the dominated convergence theorem we have
\Bea
\f{d}{dx}Z_{n}(\phi_{\eta})&=&\f{d}{dx}\E\left[e_{\eta,n}(x)\right]\\
&=&\f{d}{dx}\E\left[\exp\left(ix\sqrt{\f{b_{n}}{n}}\CN_{n}^{\circ}(\phi_{\eta})\right)\right]\\
&=&\E\left[i\sqrt{\f{b_{n}}{n}}\CN_{n}^{\circ}(\phi_{\eta})e_{\eta,n}(x)\right].
\Eea
Since by construction $\phi_{\eta}=P_{\eta}*\phi$, we have
\Bea
\CN_{n}^{\circ}(\phi_{\eta})=\f{1}{\pi}\int_{-\infty}^{\infty}\phi(\mu)\Im \g_{n}^{\circ}(z_{\mu})\;d\mu,\;\;\mbox{where} \;z_{\mu}=\mu+i\eta.
\Eea
Hereinafter, we use the finiteness of $\int_{\R}|\phi(\mu)|\;d\mu$ for $\phi\in H^{s}, s>\f{1}{2}$, when changing the order of integration. For notational convenience, from now on we will denote $e_{\eta,n}(x)$ by $e(x)$. Therefore
\Bea
\f{d}{dx}Z_n(\phi_\eta)&=&\E\left[i\sqrt{\f{b_{n}}{n}}e(x)\f{1}{\pi}\int_{-\infty}^{\infty}\phi(\mu)\Im \g_{n}^{\circ}(z_{\mu})\;d\mu\right]\\
&=&\f{1}{2\pi}\sqrt{\f{b_{n}}{n}}\int_{-\infty}^{\infty}\phi(\mu)\E\left[e(x)\text{Tr}\left(G^{\circ}(z_{\mu})-G^{\circ}(\bar z_{\mu})\right)\right]\;d\mu\\
&=&\f{1}{2\pi}\sqrt{\f{b_{n}}{n}}\int_{-\infty}^{\infty}\phi(\mu)\left(Y_{n}(z_{\mu},x)-Y_{n}(\bar z_{\mu},x)\right)\;d\mu,
\Eea
where
\bea
Y_{n}(z,x)&=&\E\left[e(x)\text{Tr}\left(G^{\circ}(z)\right)\right]\nonumber\\
&=&\E\left[e^{\circ}(x)\text{Tr} (G(z))\right]\nonumber\\
&=&n\E\left[G_{11}(z)e^{\circ}(x)\right]\nonumber\\
&=&-n\E\left[\left(A^{-1}\right)^{\circ}e_{1}(x)\right]-n\E\left[\left(A^{-1}\right)^{\circ}(e(x)-e_{1}(x))\right],\label{eqn: G_11 is written as A^-1}\\\nonumber\\
e_{1}(x)&=&\exp\left[ix\sqrt{\f{b_{n}}{n}}\left(\CN_{n-1}(\phi_\eta)\right)^{\circ}\right],\nonumber\\
\left(\CN_{n-1}(\phi_\eta)\right)^{\circ}&=&\f{1}{\pi}\int_{-\infty}^{\infty}\phi(\mu)\Im\left(\g_{n-1}(z)\right)^{\circ}\;d\mu,\nonumber\\
\g_{n-1}(z)&=&\text{Tr}G^{(1)}(z),\nonumber\\
A(z)&=&z-\f{1}{\sqrt{b_{n}}}w_{11}+\left\langle G^{(1)}m^{(1)},m^{(1)}\right\rangle,\label{definition of A}\\
m^{(1)}&=&\f{1}{\sqrt{b_{n}}}(w_{12},w_{13},\ldots,w_{1n})^T,\label{definition of m1}\\
G^{(1)}(z)&=&\left(G_{ij}^{(1)}(z)\right)_{i,j=2}^{n}=(M^{(1)}-z I)^{-1},\label{definition of G1}
\eea
and $M^{(1)}$ is the main bottom $(n-1)\times (n-1)$ minor of $M$. In the above notation $\langle \cdot,\cdot \rangle$   represents the inner product of two complex  vectors, i.e., $\langle x, y\rangle=\bar{y}^Tx$ for  $x,y \in \mathbb C^{n-1}$. The equation \eqref{eqn: G_11 is written as A^-1} follows from the Schur complement lemma, which says that
\bea\label{eqn: schur complement formula}
G_{11}(z)=\f{1}{\f{1}{\sqrt{b_{n}}}w_{11}-z-\left\langle G^{(1)}m^{(1)},m^{(1)}\right\rangle}=-\f{1}{A(z)}.
\eea
Now we rewrite
\bea\label{eqn: T_1+T_2}
\sqrt{\f{b_{n}}{n}}Y_{n}(z,x)&=&-\sqrt{nb_{n}}\E\left[\left(A^{-1}\right)^{\circ}e_{1}(x)\right]-\sqrt{nb_{n}}\E\left[\left(A^{-1}\right)^{\circ}(e(x)-e_{1}(x))\right]\nonumber \\
&=:&T_{1}+T_{2}.
\eea
Using Taylor expansion we have
\bea\label{eqn: expansion of A^{-1}}
A^{-1}=\f{1}{\E[A]}-\f{A^{\circ}}{(\E[A])^{2}}+\f{(A^{\circ})^{2}}{(\E[A])^{3}}-\f{(A^{\circ})^{3}}{(\E[A])^{4}}+\f{(A^{\circ})^{4}}{A(\E[A])^{4}}.
\eea
Therefore, we can estimate
\bea
T_{1}&=&-\sqrt{nb_{n}}\E\left[\left(A^{-1}\right)^{\circ}e_{1}(x)\right]\nonumber\\
&=&-\sqrt{nb_{n}}\E\left[\left(A^{-1}\right)e_{1}^{\circ}(x)\right]\nonumber\\
&=&-\sqrt{nb_{n}}\E\left[\left(\f{1}{\E[A]}-\f{A^{\circ}}{(\E[A])^{2}}+\f{(A^{\circ})^{2}}{(\E[A])^{3}}-\f{(A^{\circ})^{3}}{(\E[A])^{4}}+\f{(A^{\circ})^{4}}{A(\E[A])^{4}}\right)e_{1}^{\circ}(x)\right]\nonumber\\
&=&\sqrt{nb_{n}}\E\left[\left(\f{A^{\circ}}{(\E[A])^{2}}-\f{(A^{\circ})^{2}}{(\E[A])^{3}}\right)e_{1}^{\circ}(x)\right]+\sqrt{nb_{n}}\E\left[\left(\f{(A^{\circ})^{3}}{(\E[A])^{4}}-\f{(A^{\circ})^{4}}{A(\E[A])^{4}}\right)e_{1}^{\circ}(x)\right].\label{eqn: first estimate of T1}
\eea
Now we shall estimate each term individually. First of all, since $M$ is a real symmetric matrix we have
\bea
\|G(z)\|\leq\f{1}{|\Im z|}\label{eqn: norm of G is bounded by Imz},
\eea
and, in particular, $1/|A|\leq 1/|\Im z|$. It can also be checked that $1/|\E[A]|\leq 1/|\Im z|$. Hereinafter $\|X\|$ is the spectral norm of a matrix $X$. Using the above equation \eqref{eqn: norm of G is bounded by Imz} and the estimates \eqref{eqn: estimate of A^4}, \eqref{eqn: bound on variance of b_nA(z_1)A(z_2)}, we have
$$\left|\sqrt{nb_{n}}\E\left[\f{(A^{\circ})^{4}}{A(\E[A])^{4}}e_{1}^{\circ}(x)\right]\right| \leq \f{\sqrt{nb_{n}}}{|\Im z|^{5}}\E\left[|(A^{\circ})^{4}|\right]
=\f{\sqrt{nb_{n}}}{|\Im z|^{5}}O(b_{n}^{-2})
=O\left(\sqrt{\f{n}{b_{n}^{3}}}\right)\to 0,
$$
\comment{
\Bea
\left|\sqrt{nb_{n}}\E\left[\f{(A^{\circ})^{4}}{A(\E[A])^{4}}e_{1}^{\circ}(x)\right]\right|&\leq&\f{\sqrt{nb_{n}}}{|\Im z|^{5}}\E\left[|(A^{\circ})^{4}|\right]\\
&=&\f{\sqrt{nb_{n}}}{|\Im z|^{5}}O(b_{n}^{-2})\\
&=&O\left(\sqrt{\f{n}{b_{n}^{3}}}\right)\to 0,
\Eea
\Bea
\left|\sqrt{nb_{n}}\E\left[\f{(A^{\circ})^{3}}{(\E[A])^{4}}e_{1}^{\circ}(x)\right]\right|&\leq&\f{\sqrt{nb_{n}}}{|\Im z|^{4}}\E\left[|(A^{\circ})^{3}|\right]\\
&=&\f{\sqrt{nb_{n}}}{|\Im z|^{4}}O(b_{n}^{-3/2})\\
&=&O\left(\sqrt{\f{n}{b_{n}^{2}}}\right)\to 0,
\Eea}
$$\left|\sqrt{nb_{n}}\E\left[\f{(A^{\circ})^{3}}{(\E[A])^{4}}e_{1}^{\circ}(x)\right]\right| \leq \f{\sqrt{nb_{n}}}{|\Im z|^{4}}\E\left[|(A^{\circ})^{3}|\right]
 = \f{\sqrt{nb_{n}}}{|\Im z|^{4}}O(b_{n}^{-3/2})
= O\left(\sqrt{\f{n}{b_{n}^{2}}}\right)\to 0,$$

\Bea
\left|\sqrt{nb_{n}}\E\left[\f{(A^{\circ})^{2}}{(\E[A])^{3}}e_{1}^{\circ}(x)\right]\right|&\leq&\f{\sqrt{nb_{n}}}{|\Im z|^{3}}\left|\E\left[e_{1}^{\circ}(x)\E_{1}\left[(A^{\circ})^{2}\right]\right]\right|\\
&\leq&C\sqrt{\f{n}{b_{n}}}\left|\E\left[e_{1}^{\circ}(x)\left(b_{n}\E_{1}(A^{\circ})^{2}\right)\right]\right|\\
&\leq&C\sqrt{\f{n}{b_{n}}}\left[\text{Var}(e_{1}^{\circ}(x))\right]^{1/2}\left[\text{Var}\left(b_{n}\E_{1}(A^{\circ})^{2}\right)\right]^{1/2}\\
&\leq&C\sqrt{\f{n}{b_{n}}}O(b_{n}^{-1/2})\\
&=&O\left(\sqrt{\f{n}{b_{n}^{2}}}\right)\to 0, \ \mbox{as}\ n\to \infty,
\Eea
Therefore, we have
$$T_{1}=\f{\sqrt{nb_{n}}}{(\E[A])^{2}}\E\left[A^{\circ}e_{1}^{\circ}(x)\right]+O\left(\sqrt{\f{n}{b_{n}^{2}}}\right)=\f{\sqrt{nb_{n}}}{(\E[A])^{2}}\E\left[e_{1}^{\circ}(x)\E_{1}(A^{\circ})\right]+O\left(\sqrt{\f{n}{b_{n}^{2}}}\right).
$$
Now
\Bea
A^{\circ}&=&-\f{1}{\sqrt{b_{n}}}w_{11}+\f{1}{b_{n}}\sum_{\stackrel{i\neq j}{i,j \in I_1}}G_{ij}^{(1)}w_{1i}w_{1j}+\f{1}{b_{n}}\sum_{i\in I_1}\left(G_{ii}^{(1)}w_{1i}^{2}-\E[G_{ii}^{(1)}]\right),
\Eea
where $I_1=\{1< j \leq n: (1,j)\in I_n\}$.
Therefore,
\Bea
\E_{1}[A^{\circ}(z)]=\f{1}{b_{n}}\sum_{i\in I_1}\left(G_{ii}^{(1)}-\E[G_{ii}^{(1)}]\right)
\Eea
and hence
\bea
T_{1}&=&\f{\sqrt{nb_{n}}}{(\E[A])^{2}}\E\left[e_{1}^{\circ}(x)\E_{1}(A^{\circ})\right]+O\left(\sqrt{\f{n}{b_{n}^{2}}}\right)\nonumber\\
&=&\f{\sqrt{nb_{n}}}{{(\E[A])^{2}}}\E\left[e_{1}^{\circ}(x)\f{1}{b_{n}}\sum_{i\in I_1}(G_{ii}^{(1)}-\E[G_{ii}^{(1)}])\right]+O\left(\sqrt{\f{n}{b_{n}^{2}}}\right)\nonumber\\
&=&\f{\sqrt{nb_{n}}}{{(\E[A])^{2}}}\;2\E\left[(G_{22}^{(1)})^{\circ}e_{1}^{\circ}(x)\right]+O\left(\sqrt{\f{n}{b_{n}^{2}}}\right)\nonumber\\
&=&\f{\sqrt{nb_{n}}}{{(\E[A])^{2}}}\;\f{2}{n}\;\E[\g_{n-1}^{\circ}e_{1}^{\circ}(x)]+O\left(\sqrt{\f{n}{b_{n}^{2}}}\right)\nonumber\\
&=&\sqrt{\f{b_{n}}{n}}\f{2}{(\E[A])^{2}}\E\left[\g_{n-1}^{\circ}e_{1}(x)\right]+O\left(\sqrt{\f{n}{b_{n}^{2}}}\right).\label{eqn: second estimate of T1}
\eea
Hereinafter, all bounds (implicitly) depending on $z$ hold uniformly on the set $\{\mu+i\eta:\mu\in\R\}$ for any given $\eta>0$. Now
\Bea
\left|\E\left[\g_{n-1}^{\circ}e_{1}(x)\right]-\E\left[\g_{n}^{\circ}e(x)\right]\right|&=&\left|\E\left[\g_{n-1}^{\circ}e_{1}(x)\right]-\E\left[\g_{n}^{\circ}e_{1}(x)\right]+\E\left[\g_{n}^{\circ}e_{1}(x)\right]-\E\left[\g_{n}^{\circ}e(x)\right]\right|\\
&\leq&\left(\E\left[\left|\g_{n-1}^{\circ}-\g_{n}^{\circ}\right|^{4}\right]\right)^{1/4}+\left|\E\left[\g_{n}^{\circ}(e_{1}(x)-e(x))\right]\right|\\
&=&O(b_{n}^{-1/2})+\left|\E\left[\g_{n}^{\circ}(e_{1}(x)-e(x))\right]\right|.
\Eea
The last equality follows from \eqref{eqn: fourth moment of g_n^1-g_n}. We estimate
\bea
e(x)-e_{1}(x)&=&\exp\left[ix\sqrt{\f{b_{n}}{n}}\CN_{n}^{\circ}(\phi_\eta)\right]-\exp\left[ix\sqrt{\f{b_{n}}{n}}\CN_{n-1}^{\circ}(\phi_\eta)\right]\nonumber\\
&=&\left(\exp\left[ix\sqrt{\f{b_{n}}{n}}\CN_{n}^{\circ}(\phi_\eta)-ix\sqrt{\f{b_{n}}{n}}\CN_{n-1}^{\circ}(\phi_\eta)\right]-1\right)e_{1}(x)\nonumber\\
&=&ix\sqrt{\f{b_{n}}{n}}\left(\CN_{n}^{\circ}(\phi_\eta)-\CN_{n-1}^{\circ}(\phi_\eta)\right)e_{1}(x)+\f{b_{n}}{n}O\left(x^{2}\left(\CN_{n}^{\circ}(\phi_\eta)-\CN_{n-1}^{\circ}(\phi_\eta)\right)^{2}e_{1}(x)\right)\nonumber\\
&=&\f{ix}{\pi}\sqrt{\f{b_{n}}{n}}\int_{-\infty}^{\infty}\left[\phi(\mu)\Im\left(\g_{n}^{\circ}-\g_{n-1}^{\circ}\right)e_{1}(x)+\sqrt{\f{b_{n}}{n}}\phi(\mu)O(\g_{n}^{\circ}-\g_{n-1}^{\circ})^{2}\right]\;d\mu.\label{eqn: estimete of e_1(x)-e(x)}
\eea
Therefore
\Bea
\E\left[\g_{n}^{\circ}(e(x)-e_{1}(x))\right]=\E\left[\f{ix}{\pi}\sqrt{\f{b_{n}}{n}}\int_{-\infty}^{\infty}\phi(\mu)\left[\Im\left(\g_{n}^{\circ}-\g_{n-1}^{\circ}\right)e_{1}(x)\g_{n}^{\circ}+\sqrt{\f{b_{n}}{n}}\g_{n}^{\circ}O\left(\g_{n}^{\circ}-\g_{n-1}^{\circ}\right)^{2}\right]\;d\mu\right].
\Eea
Using estimates \eqref{eqn: bound on the variance of gamma_n} and \eqref{eqn: fourth moment of g_n^1-g_n}, we have
$$\left|\E\left[\Im\left(\g_{n}^{\circ}-\g_{n-1}^{\circ}\right)e_{1}(x)\g_{n}^{\circ}\right]\right|\leq \left(\E[|\g_{n}^{\circ}|^{2}]\right)^{1/2}\left(\E\left[\left|e_{1}(x)\Im\left(\g_{n}^{\circ}-\g_{n-1}^{\circ}\right)\right|^{2}\right]\right)^{1/2}
=O\left(\sqrt{\f{n}{b_{n}}}\sqrt{\f{1}{b_{n}}}\right).
$$
Similarly,
\Bea
\E\left[\g_{n}^{\circ}O\left(\g_{n}^{\circ}-\g_{n-1}^{\circ}\right)^{2}\right]&=&O\left(\sqrt{\f{n}{b_{n}}}\f{1}{b_{n}}\right).
\Eea
Therefore,
\Bea
\left|\E\left[\g_{n-1}^{\circ}e_{1}(x)\right]-\E\left[\g_{n}^{\circ}e(x)\right]\right|=O\left(\f{1}{\sqrt{b_{n}}}\right).
\Eea
From the equation \eqref{eqn: second estimate of T1} and the above estimates we have
\bea
T_{1}&=&\sqrt{\f{b_{n}}{n}}\f{2}{(\E[A])^{2}}\E\left[\g_{n-1}^{\circ}e_{1}(x)\right]+O\left(\f{\sqrt{n}}{b_{n}}\right)\nonumber\\
&=&\sqrt{\f{b_{n}}{n}}\f{2}{(\E[A])^{2}}\left[\E[\g_{n}^{\circ}e(x)]+O\left(\f{1}{\sqrt{b_{n}}}\right)\right]+O\left(\f{\sqrt{n}}{b_{n}}\right)\nonumber\\
&=&\sqrt{\f{b_{n}}{n}}\f{2}{(\E[A])^{2}}Y_{n}(z,x)+O\left(\f{\sqrt{n}}{b_{n}}\right).\label{eqn: third estimate of T1}
\eea
Now consider $T_{2}$. Using \eqref{eqn: estimete of e_1(x)-e(x)} and \eqref{eqn: trace of difference of resolvents interms of A and B} we have
\Bea
T_{2}&=&-\sqrt{nb_{n}}\E\left[(A^{-1})^{\circ}(e(x)-e_{1}(x))\right]\\
&=&-\f{ixb_{n}}{\pi}\E\left[(A^{-1})^{\circ}\int_{-\infty}^{\infty}\phi(\mu)\Im\left(\g_{n}^{\circ}-\g_{n-1}^{\circ}\right)e_{1}(x)\;d\mu\right]\\
&&-\f{1}{\pi}\sqrt{\f{b_{n}^{3}}{n}}\E\left[(A^{-1})^{\circ}\int_{-\infty}^{\infty}\phi(\mu)O(\g_{n}^{\circ}-\g_{n-1}^{\circ})^{2}\;d\mu\right]\\
&=&-\f{ixb_{n}}{\pi}\int_{-\infty}^{\infty}\phi(\mu)\E\left[e_{1}(x)(A^{-1})^{\circ}\Im(\g_{n}^{\circ}-\g_{n-1}^{\circ})\right]\;d\mu+\sqrt{\f{b_{n}^{3}}{n}}O\left(\f{1}{b_{n}}\right)\\
&=&-\f{ixb_{n}}{\pi}\int_{-\infty}^{\infty}\phi(\mu)\E\left[e_{1}(x)(A^{-1})^{\circ}\Im\left(\g_{n}^{\circ}-\g_{n-1}^{\circ}\right)\right]\;d\mu+O\left(\sqrt{\f{b_{n}}{n}}\right)\\
&=&-\f{ixb_{n}}{\pi}\int_{-\infty}^{\infty}\phi(\mu)\E\left[e_{1}(x)(A^{-1})^{\circ}\Im\left(\g_{n}-\g_{n-1}\right)^{\circ}\right]\;d\mu+O\left(\sqrt{\f{b_{n}}{n}}\right)\\
&=&\f{ixb_{n}}{\pi}\int_{-\infty}^{\infty}\phi(\mu)\E\left[e_{1}(x)(A^{-1})^{\circ}(z)\Im\left(\f{1+B(z_{\mu})}{A(z_{\mu})}\right)^{\circ}\right]\;d\mu+O\left(\sqrt{\f{b_{n}}{n}}\right)\\
&=&T_{21}-T_{22}+O\left( \sqrt{\f{b_{n}}{n}}\right),
\Eea
where $B(z)=\left\langle G^{(1)}G^{(1)}m^{(1)},m^{(1)}\right\rangle$ and
\Bea
T_{21}&=&\f{xb_{n}}{2\pi}\int_{-\infty}^{\infty}\phi(\mu)\E\left[e_{1}(x)(A^{-1})^{\circ}(z)\left(\f{1+B(z_{\mu})}{A(z_{\mu})}\right)^{\circ}\right]d\mu,\\
T_{22}&=&\f{xb_{n}}{2\pi}\int_{-\infty}^{\infty}\phi(\mu)\E\left[e_{1}(x)(A^{-1})^{\circ}(z)\left(\f{\overline{1+B(z_{\mu})}}{\overline{A(z_{\mu})}}\right)^{\circ}\right]d\mu.
\Eea
Using $\Im\left\langle G^{(1)}m^{(1)},m^{(1)}\right\rangle=\Im z\left\langle G^{(1)*}G^{(1)}m^{(1)},m^{(1)} \right\rangle$, it can be easily verified that 
\bea\label{eqn: bounds of B/A, 1/EA, and EB by 1/Imz}
\left|\f{B(z)}{A(z)}\right|\leq \f{1}{|\Im z|},\;\f{1}{|\E[A(z)]|}\leq \f{1}{|\Im z|},\;\text{and}\;|\E[B(z)]|\leq \f{2}{|\Im z|^{2}}.
\eea
Applying $A^{-1}=\f{1}{\E[A]}-\f{A^{\circ}}{(\E[A])^{2}}+\f{(A^{\circ})^{2}}{A(\E[A])^{2}}$ to $A^{-1}(z),\ A^{-1}(z_\mu)$ and using \eqref{eqn: estimate of A^4}, we get

\Bea
&&b_{n}\E\left[e_{1}(x)(A^{-1})^{\circ}(z)\left(\f{1+B(z_{\mu})}{A(z_{\mu})}\right)^{\circ}\right]\\
&=&b_n\E\left[e_{1}(x)\left\{  \frac{A^{\circ}(z)}{\E^2[A(z)]}\left(-\frac{B^{\circ}(z_\mu)}{\E[A(z_\mu)]} +\frac{(1+B(z_\mu))A^{\circ}(z_\mu)}{\E^2[A(z_\mu)]}-\frac{\E[B(z_\mu)A^{\circ}(z_\mu)]}{\E^2[A(z_\mu)]}\right)\right\}\right]+O(b_n^{-1/2})\\
&=&b_n\E\left[e_{1}(x)\left\{ \frac{A^{\circ}(z)}{\E^2[A(z)]} \left(-\frac{B^{\circ}(z_\mu)}{\E[A(z_\mu)]} +\frac{(1+\E[B(z_\mu)])A^{\circ}(z_\mu)}{\E^2[A(z_\mu)]}+\frac{(B^{\circ}(z_\mu)A^{\circ}(z_\mu))^{\circ}}{\E^2[A(z_\mu)]}\right)\right\}\right]+O(b_n^{-1/2})\\
&=&\frac{(1+\E [B(z_\mu)])}{\E^2[A(z)]\E^2[A(z_\mu)]}\E \left[e_1(x)b_n A^{\circ}(z) A^{\circ}(z_\mu)\right]-\frac{\E \left[e_1(x) b_n{A^{\circ}(z)} {B^{\circ}(z_\mu)} \right]}{\E^2[A(z)]\E[A(z_\mu)]}+O(b_n^{-1/2})\\
&=&\frac{(1+\E [B(z_\mu)])}{\E^2[A(z)]\E^2[A(z_\mu)]}\E \left[e_1(x)\E_1\left(b_n A^{\circ}(z)A^{\circ}(z_\mu)\right)\right]-\f{\E\left[e_1(x)\E_1\left[ b_nA^{\circ}(z) B^{\circ}(z_\mu)\right]\right]}{\E^2[A(z)] \E[A(z_\mu)]}+O(b_n^{-1/2}).
\Eea
Using \eqref{eqn: bound on variance of b_nA(z_1)A(z_2)}, from the last expression  we get
\bea\label{eqn: e_1(x)b_nA(z)B(z)}
&&b_{n}\E\left[e_{1}(x)(A^{-1})^{\circ}(z)\left(\f{1+B(z_{\mu})}{A(z_{\mu})}\right)^{\circ}\right]\nonumber \\
&=&\frac{(1+\E B(z_\mu))}{\E^2[A(z)]\E^2[A(z_\mu)]}\E [e_1(x)]\E \left[b_n A^{\circ}(z)A^{\circ}(z_\mu)\right]-\f{\E[e_1(x)] \E\left[b_n A^{\circ}(z) B^{\circ}(z_\mu)\right]}{\E^2[A(z)] \E[A(z_\mu)]}+O(b_n^{-1/2}).
\eea
Define
\Bea
D_{n}(z,z_{\mu})
&=&\f{(1+\E[B(z_{\mu})])\E[b_{n}\E_{1}\left\{A^{\circ}(z)A^{\circ}(z_{\mu})\right\}]}{\E^{2}[A(z)]\E^{2}[A(z_{\mu})]}-\f{\E[b_{n}\E_{1}\{A^{\circ}(z)B^{\circ}(z_{\mu})\}]}{\E^{2}[A(z)]\E[A(z_{\mu})]}.
\Eea
Also, using \eqref{eqn: estimete of e_1(x)-e(x)} and \eqref{eqn: fourth moment of g_n^1-g_n}, we have
\Bea
\E[e(x)]-\E[e_{1}(x)]&=&\E\left[\f{ix}{\pi}\sqrt{\f{b_{n}}{n}}\int_{-\infty}^{\infty}\phi(\mu)\Im\left(\g_{n}^{\circ}-\g_{n-1}^{\circ}\right)e_{1}(x)\ d\mu+\f{b_{n}}{n}x^2\int_{-\infty}^{\infty} \phi(\mu)O(\g_{n}^{\circ}-\g_{n-1}^{\circ})^{2}\;d\mu\right]\\
&=&O(n^{-1/2})+O(n^{-1}).
\Eea
Therefore 
\bea
\E[e_{1}(x)]=Z_{n}(\phi_{\eta})+O({n}^{-1/2}).\label{eqn: expectation of e1 in terms of Z(phieta)}
\eea
Combining \eqref{eqn: T_1+T_2}, \eqref{eqn: third estimate of T1}, \eqref{eqn: e_1(x)b_nA(z)B(z)}, and \eqref{eqn: expectation of e1 in terms of Z(phieta)}, we get
\Bea
\sqrt{\f{b_{n}}{n}}Y_{n}(z,x)&=&T_{1}+T_{2}\\
&=&\f{2}{\E^{2}[A]}\sqrt{\f{b_{n}}{n}}Y_{n}(z,x)+\f{x}{2\pi}\E[e_{1}(x)]\int_{-\infty}^{\infty} [D_{n}(z,z_{\mu})-D_{n}(z,\bar z_{\mu}) ]\phi(\mu)\;d\mu +O(b_n^{-1/2})\\
&=&\f{2}{\E^{2}[A]}\sqrt{\f{b_{n}}{n}}Y_{n}(z,x)+\f{x}{2\pi}Z_{n}(\phi_\eta{})\int_{-\infty}^{\infty} [D_{n}(z,z_{\mu})-D_{n}(z,\bar z_{\mu})] \phi(\mu)\;d\mu + o(1)
\Eea
\Bea
&\approx &2f^{2}(z)\tilde Y_{n}(z,x)+\f{x}{2\pi}Z_{n}(\phi_\eta{} ) \int_{-\infty}^{\infty}[D_{n}(z,z_{\mu})-D_{n}(z,\bar z_{\mu})] \phi(\mu)\;d\mu + o(1),
\Eea
where $\tilde Y_{n}(z,x)=\sqrt{\f{b_{n}}{n}}Y_{n}(z,x)$. Therefore,
\Bea
\tilde Y_{n}(z,x)=Z_{n}(\phi_{\eta})\f{x}{2\pi}\int_{-\infty}^{\infty}\left(C_{n}(z,z_{\mu})-C_{n}(z,\bar z_{\mu})\right)\phi(\mu)\;d\mu+ o(1)
\Eea
uniformly in $z$ with $\Im z=\eta$, where $C_{n}(z,z_{\mu})=\f{D_{n}(z,z_{\mu})}{1-2f^{2}(z)}$ and $f(z)$ is given in \eqref{eqn: asymptotic of f_n(z)}. Hence

\Bea
\f{d}{dx}Z_{n}(\phi_{\eta})&=&\f{1}{2\pi}\int_{-\infty}^{\infty}\phi(\mu)\left(\tilde Y_{n}(z_{\mu},x)-\tilde Y_{n}(\bar z_{\mu},x)\right)\;d\mu\\
&=&\f{1}{2\pi}\int_{-\infty}^{\infty}\phi(\mu_{1})\left[\f{x}{2\pi}Z_{n}(\phi_{\eta})\int_{-\infty}^{\infty}\phi(\mu_{2})\left(C_{n}(z_{\mu_{1}},z_{\mu_{2}})-C_{n}(z_{\mu_{1}},\bar z_{\mu_{2}})\right)d\mu_{2} \right.\\
&&\left.-\f{x}{2\pi}Z_{n}(\phi_{\eta})\int_{-\infty}^{\infty}\phi(\mu_{2})\left(C_{n}(\bar z_{\mu_{1}},z_{\mu_{2}})-C_{n}(\bar z_{\mu_{1}},\bar z_{\mu_{2}})\right)d\mu_{2} \right]d\mu_{1}+o(1)\\
&=&-\f{x}{4\pi^{2}}Z_{n}(\phi_{\eta})\int_{-\infty}^{\infty}\int_{-\infty}^{\infty}\phi(\mu_{1})\phi(\mu_{2})\left[C_{n}(z_{\mu_{1}},\bar z_{\mu_{2}})+C_{n}(\bar z_{\mu_{1}},z_{\mu_{2}})\right.\\
&&\left.-C_{n}(z_{\mu_{1}},z_{\mu_{2}})-C_{n}(\bar z_{\mu_{1}},\bar z_{\mu_{2}})\right]d\mu_{2}d\mu_{1}+o(1)\\
&=&-xZ_{n}(\phi_{\eta})V_{n}(\phi,\eta)+o(1).
\Eea
To find the limit of $V_{n}(\phi,\eta)$, we shall calculate limit of $[C_{n}(z_{\mu_{1}},\bar z_{\mu_{2}})+C_{n}(\bar z_{\mu_{1}},z_{\mu_{2}})-C_{n}(z_{\mu_{1}},z_{\mu_{2}})-C_{n}(\bar z_{\mu_{1}},\bar z_{\mu_{2}})]$ as $n\to\infty$.
Using \eqref{eqn: b_nA(z_1)A(z_2)} and \eqref{A(z_1)B(z_2)},
\Bea
D_{n}(z,z_{\mu})&=&\f{(1+\E[B(z_{\mu})])\E[b_{n}\E_{1}\left\{A^{\circ}(z)A^{\circ}(z_{\mu})\right\}]}{\E^{2}[A(z)]\E^{2}[A(z_{\mu})]}-\f{\E[b_{n}\E_{1}\{A^{\circ}(z)B^{\circ}(z_{\mu})\}]}{\E^{2}[A(z)]\E[A(z_{\mu})]}\\
&=&\f{1+\E[B(z_{\mu})]}{\E^{2}[A(z)]\E^{2}[A(z_{\mu})]}\E\Big[\f{2}{b_{n}}\sum_{i,j\in I_{1}}G_{ij}^{(1)}(z)G_{ij}^{(1)}(z_{\mu})+\s^{2}+\f{\kappa_{4}}{b_{n}}\sum_{i\in I_{1}}G_{ii}^{(1)}(z)G_{ii}^{(1)}(z_{\mu})\\
&&+\f{1}{b_{n}}\widetilde{\g_{n-1}}(z)\widetilde{\g_{n-1}}(z_{\mu})])\Big]-\f{1}{\E^{2}[A(z)]\E[A(z_{\mu})]}\f{d}{dz_{\mu}}\E\Big[\f{2}{b_{n}}\sum_{i,j\in I_{1}}G_{ij}^{(1)}(z)G_{ij}^{(1)}(z_{\mu})\\
&& +\s^{2}+\f{\kappa_{4}}{b_{n}}\sum_{i\in I_{1}}G_{ii}^{(1)}(z)G_{ii}^{(1)}(z_{\mu})+\f{1}{b_{n}}\widetilde{\g_{n-1}}(z)\widetilde{\g_{n-1}}(z_{\mu})])\Big].
\Eea
Now using \eqref{eqn: variance of sum of Gii is small}, we get
\Bea
\left|\E\left[\f{1}{b_{n}}\widetilde{\g_{n-1}}(z)\widetilde{\g_{n-1}}(z_{\mu})])\right]\right|\leq\f{1}{b_{n}}\sqrt{\text{Var}\sum_{i\in I_{1}}G_{ii}^{(1)}}\sqrt{\text{Var}\sum_{i\in I_{1}}G_{ii}^{(1)}}=O\left(\f{1}{b_{n}}\right).
\Eea
Letting $n\to\infty$, using \eqref{eqn: limit of A(z),B(z)} we have
\bea
\lim_{n\to\infty}D_{n}(z,z_{\mu})&=&f^{2}(z)f^{2}(z_{\mu})(1+2f'(z_{\mu}))\left[\lim_{n\to\infty}\E[T_{n}]+\s^{2}+\kappa_{4}\lim_{n\to\infty}\f{1}{b_{n}}\sum_{i\in I_{1}}\E\left[G_{ii}^{(1)}(z)G_{ii}^{(1)}(z_{\mu})\right]\right]\nonumber\\
&&+f^{2}(z)f(z_{\mu})\f{d}{dz_{\mu}}\left[\lim_{n\to\infty}\E[T_{n}]+\kappa_{4}\lim_{n\to\infty}\f{1}{b_{n}}\sum_{i\in I_{1}}\E\left[G_{ii}^{(1)}(z)G_{ii}^{(1)}(z_{\mu})\right]\right],\label{eqn: limit of Dn}
\eea
where
\Bea
T_{n}=\f{2}{b_{n}}\sum_{i,j\in I_{1}}G_{ij}^{(1)}(z)G_{ij}^{(1)}(z_{\mu}).
\Eea
Since $\text{Var}(G_{ii})=O(1/b_{n})$ (see \eqref{eqn: variance of sum of Gii is small}), we have
\Bea
\lim_{n\to\infty}\f{1}{b_{n}}\sum_{i\in I_{1}}\E\left[G_{ii}^{(1)}(z)G_{ii}^{(1)}(z_{\mu})\right]=\lim_{n\to\infty}\f{1}{b_{n}}\sum_{i\in I_{1}}\E\left[G_{ii}^{(1)}(z)\right]\E\left[G_{ii}^{(1)}(z_{\mu})\right]=2f(z)f(z_{\mu}).
\Eea
We shall show in the appendix \eqref{proof of ET_n} that
\bea\label{eqn: limit of ET_n}
\lim_{n\to\infty}\E[T_{n}]=\f{1}{4\pi^{3}}\int_{-2\sqrt{2}}^{2\sqrt{2}}\int_{-2\sqrt{2}}^{2\sqrt{2}}\f{\sqrt{8-x^{2}}\sqrt{8-y^{2}}}{(x-z)(y-z_{\mu})}F(x,y)\1_{\{x\neq y\}}\;dxdy,
\eea
where
\Bea
F(x,y)=2\int_{-\infty}^{\infty}\f{u-u^{3}}{2(1-u^{2})^{2}+u^{2}(x^{2}+y^{2})-u(1+u^{2})xy}\;ds,
\Eea
where $u=\f{\sin s}{s}$.
Therefore
\Bea
\lim_{n\to\infty}C_{n}(z_{\mu_{1}},z_{\mu_{2}})&=&\f{1}{1-2f^{2}(z_{\mu_{1}})}\left[f^{2}(z_{\mu_{1}})f^{2}(z_{\mu_{2}})(1+2f'(z_{\mu_{2}}))\lim_{n\to\infty}\E[T_{n}]+f^{2}(z_{\mu_{1}})f(z_{\mu_{2}})\lim_{n\to\infty}\f{d}{dz_{\mu_2}}\E[T_{n}]\right.\\
&& +\s^{2}f^{2}(z_{\mu_{1}})f^{2}(z_{\mu_{2}})(1+2f'(z_{\mu_{2}}))+2\kappa_{4}\Bigl\{f^{3}(z_{\mu_{1}})f^{3}(z_{\mu_{2}})(1+2f'(z_{\mu_{2}}))\\
&& +f^{3}(z_{\mu_{1}})f(z_{\mu_{2}})f'(z_{\mu_{2}})\Bigr\}\biggr].
\Eea
Hence
\Bea
V(\phi)&=&\lim_{\eta\downarrow 0}\lim_{n\to\infty}V_{n}(\phi,\eta)\\
&=&\f{\kappa_{4}}{16\pi^{2}}\left(\int_{-2\sqrt{2}}^{2\sqrt{2}}\f{4-\mu^{2}}{\sqrt{8-\mu^{2}}}\phi(\mu)\;d\mu\right)^{2}+\f{\s^{2}}{16\pi^{2}}\left(\int_{-2\sqrt{2}}^{2\sqrt{2}}\f{\mu\phi(\mu)}{\sqrt{8-\mu^{2}}}\;d\mu\right)^{2}\\
&&+\int_{-2\sqrt{2}}^{2\sqrt{2}}\int_{-2\sqrt{2}}^{2\sqrt{2}}\sqrt{(8-x^{2})(8-y^{2})}F(x,y) \int_{-2\sqrt{2}}^{2\sqrt{2}}\int_{-2\sqrt{2}}^{2\sqrt{2}}\f{\mu_{1}\phi(\mu_{1})}{(x-\mu_{1})\sqrt{8-\mu_{1}^{2}}}\f{\mu_{2}\phi(\mu_{2})}{(x-\mu_{2})^2\sqrt{8-\mu_{2}^{2}}}d\mu_{1}d\mu_{2}\ dx dy.
\Eea
This completes the proof of \eqref{defn: characteristic function corresponding to phieta} and the proof of Theorem \ref{thm: Main theorem 1}.
\end{proof}
\section{Appendix}

\begin{proof}[Proof of Proposition \ref{prop: bound on Var(gamma)}:]
Let us denote the averaging with respect to $\{w_{ij}; 1\leq i\leq k\ \mbox{or}\ 1\leq j\leq n\}$ by $\E_{\leq k}$ and the averaging with respect to $\{w_{kj}; 1\leq j\leq n\}$ by $\E_k$. Using the martingale difference technique (see \cite{dharmadhikari1968bounds}), we have
\Bea
\text{Var}\{\g_{n}\}&\leq& \sum_{k=1}^{n}\E\left[\left|\E_{\leq k-1}[\g_{n}]-\E_{\leq k}[\g_{n}]\right|^{2}\right]\\
&=&\sum_{k=1}^{n}\E\left[\left|\E_{\leq k-1}\left[\g_{n}-\E_{\leq k}[\g_{n}]\right]\right|^{2}\right]\\
&\leq &\sum_{k=1}^{n}\E\left[\E_{\leq k-1}\left|\g_{n}-\E_{k}[\g_{n}]\right|^{2}\right]\\
&=&\sum_{k=1}^{n}\E\left[\left|\g_{n}-\E_{k}[\g_{n}]\right|^{2}\right].
\Eea
Note that
\Bea
\E\left[\left|\g_{n}-\E_{1}[\g_{n}]\right|^{2}\right]&=&\E\left[\left|\text{Tr}(G)-\E_{1}[\text{Tr}(G)]\right|^{2}\right]\\
&=&\E\left[\left|\text{Tr}(G)-\E_{1}[\text{Tr}(G)]+\text{Tr}(G^{(1)})-\text{Tr}(G^{(1)})\right|^{2}\right]\\
&=&\E\left[\left|\text{Tr}(G-G^{(1)})-\E_{1}\left[\text{Tr}(G-G^{(1)})\right]\right|^{2}\right].
\Eea
From \eqref{difference between the traces of G and G^1} we have
\bea
\text{Tr}(G-G^{(1)})&=&-\f{1+B(z)}{A(z)}\label{eqn: trace of difference of resolvents interms of A and B}
\eea
where $A(z)=-G_{11}^{-1}$, $B(z)=\left\langle G^{(1)}G^{(1)}m^{(1)},m^{(1)}\right\rangle$, and $G^{(1)}$ is defined in \eqref{definition of G1}, and $m^{(1)}=\frac{1}{\sqrt{b_{n}}}(w_{12},w_{13},\ldots,w_{1n})^T$. Indeed,
\Bea
\E\left[\left|\g_{n}-\E_{1}[\g_{n}]\right|^{2}\right]&\leq&\E\left[\left|\f{1+B(z)}{A(z)}-\E_1\left[\f{1+B(z)}{A(z)}\right]\right|^{2}\right]\\
&\leq&2\E\left[\left|\f{1}{A(z)}-\E_{1}\left[\f{1}{A(z)}\right]\right|^{2}\right]+2\E\left[\left|\f{B(z)}{A(z)}-\E_{1}\left[\f{B(z)}{A(z)}\right]\right|^2\right].
\Eea
Now, by \eqref{eqn: norm of G is bounded by Imz} and \eqref{eqn: bounds of B/A, 1/EA, and EB by 1/Imz},
\Bea
\E_{1}\left[\left|\f{B(z)}{A(z)}-\E_{1}\left[\f{B(z)}{A(z)}\right]\right|^{2}\right]&\leq& \E_{1}\left[\left|\f{B(z)}{A(z)}-\f{\E_{1}[B(z)]}{E_{1}[A(z)]}\right|^{2}\right]\\
&\leq&\E_{1}\left[\left|\f{B_{1}^{\circ}}{\E_{1}[A]}-\f{A_{1}^{\circ}}{\E_{1}[A]}\f{B}{A}\right|^{2}\right]\\
&\leq&2\E_{1}\left[\left|\f{B_{1}^{\circ}}{\E_{1}[A]}\right|^{2}\right]+\f{2}{|\Im z|^{2}}\E_{1}\left[\left|\f{A_{1}^{\circ}}{\E_{1}[A]}\right|^{2}\right],
\Eea
where $A_{1}^{\circ}=A-\E_{1}[A]$. So it is enough to estimate $\E_{1}\left[\left|\f{A_{1}^{\circ}}{\E_{1}[A]}\right|^{2}\right]$ and $\E_{1}\left[\left|\f{B_{1}^{\circ}}{\E_{1}[A]}\right|^{2}\right]$. Note that
\Bea
A&=&z-\f{1}{\sqrt{b_{n}}}w_{11}+\left\langle G^{(1)}m^{(1)},m^{(1)}\right\rangle\\
A_{1}^{\circ}&=&-\f{1}{\sqrt{b_{n}}}w_{11}+\f{1}{b_{n}}\sum_{\stackrel{i\neq j}{i,j\in I_{1}}}G_{ij}^{(1)}w_{1i}w_{1j}+\f{1}{b_{n}}\sum_{i\in I_{1}}G_{ii}^{(1)}(w_{1i}^{2})^{\circ}.
\Eea
Therefore,
\bea
\E_{1}\left[\left| A_{1}^{\circ}\right|^{2}\right]&=&\E_{1}\left[\f{1}{b_{n}}w_{11}^{2}+\f{1}{b_{n}^{2}}\sum_{\stackrel{i\neq j}{i,j\in I_{1}}}G_{ij}^{(1)}w_{1i}w_{1j}\sum_{\stackrel{k\neq l}{k,l\in I_{1}}}\overline{G_{kl}^{(1)}}w_{1k}w_{1l}+\frac{1}{b_n^2}\sum_{i\in I_1}G_{ii}^{(1)}(w_{1i}^2)^{\circ}\sum_{l\in I_1}\overline{G_{ll}^{(1)}}(w_{1l}^2)^{\circ}\right]\nonumber\\
&=&\f{\s^{2}}{b_{n}}+\f{2}{b_{n}^{2}}\sum_{\stackrel{i\neq j}{i,j\in I_{1}}}|G_{ij}^{(1)}|^{2}+\f{\mu_{4}-1}{b_{n}^{2}}\sum_{i\in I_{1}}|G_{ii}^{(1)}|^{2}\nonumber\\
&\leq & \f{\s^{2}}{b_{n}}+\f{2}{b_{n}^{2}} \f{2b_{n}}{|\Im z|^{2}}+\f{\mu_{4}-1}{b_{n}^{2}} \f{2b_{n}}{|\Im z|^{2}}\nonumber\\
&\leq&\f{1}{b_{n}}\left(\s^{2}+\f{2+2\mu_{4}}{|\Im z|^{2}}\right)\label{estimate of A1 as O(1/b)}.
\eea

Now we want to estimate $\E_{1}\left[\left|B_{1}^{\circ}\right|^{2}\right]$, where $B=\left\langle G^{(1)}G^{(1)}m^{(1)},m^{(1)}\right\rangle=\left\langle H^{(1)}m^{(1)},m^{(1)}\right\rangle$, and $B_{1}^{\circ}=B-\E_{1}[B]$. Therefore,
\Bea
\E_{1}[B]=\f{1}{b_{n}}\sum_{i\in I_{1}}H_{ii}^{(1)}=\f{1}{b_{n}}\sum_{i\in I_{1}}\sum_{j=2}^{n}\left(G_{ij}^{(1)}\right)^{2},
\Eea
and
\Bea
B_{1}^{\circ}=\f{1}{b_{n}}\sum_{\stackrel{i\neq j}{i,j\in I_{1}}}H_{ij}^{(1)}w_{1i}w_{1j}+\f{1}{b_{n}}\sum_{i\in I_{1}}H_{ii}^{(1)}\left(w_{1i}^{2}\right)^{\circ}.
\Eea
Let us call $C_{0}=\E\left[(w_{1i}^{2})^{\circ}\right]^{2}$. Then
\Bea
\E_{1}[|B_{1}^{\circ}|^{2}]&=&\f{1}{b_{n}^{2}}\sum_{\stackrel{i\neq j}{i,j\in I_{1}}}|H_{ij}^{(1)}|^{2}+\f{C_{0}}{b_{n}^{2}}\sum_{i\in I_{1}}|H_{ii}^{(1)}|^{2}\\
&=&\f{1}{b_{n}^{2}}\sum_{\stackrel{i\neq j}{i,j\in I_{1}}}\left|\sum_{k=2}^{n}G_{ik}^{(1)}G_{kj}^{(1)}\right|^{2}+\f{C_{0}}{b_{n}^{2}}\sum_{i\in I_{1}}\left|\sum_{k=2}^{n}G_{ik}^{(1)}G_{ki}^{(1)}\right|^{2}\\
&\leq&\f{1}{b_{n}^{2}}\sum_{i\in I_{1}}\| (G^{(1)})^{2}\|^{2}+\f{C_{0}}{b_{n}^{2}} 2b_{n}\| (G^{(1)})^{2}\|^{2}\\
&=&\f{2}{b_{n}}\f{1}{|\Im z|^{4}}+\f{2C_{0}}{b_{n}}\f{1}{|\Im z|^{4}}\\
&=&\f{2(1+C_{0})}{b_{n}|\Im z|^{4}}.
\Eea
We also have
\Bea
\E\left[\left|\f{1}{A(z)}-\E_{1}\left[\f{1}{A(z)}\right]\right|^{2}\right]\leq\f{1}{|\Im z|^{2}}\E\left[\left|\f{A_{1}^{\circ}}{\E_{1}[A]}\right|^{2}\right] 
\Eea
Note that $\E_{1}[A]=z+\f{1}{b_{n}}\sum_{i\in I_{1}}G_{ii}^{(1)}$. Since $\Im G_{ii}^{(1)}>0$, we have $|\E_{1}[A]|\geq |\Im A|\geq y$. Also we know that $|G_{ii}^{(1)}|\leq 1/|\Im z|=1/y$. Therefore $|\E_{1}[A]|\geq |x|-\f{2}{y}$. Combining these we have 
\Bea
|\E_{1}[A]|>\max\left\{y,|x|-\f{2}{y}\right\}.
\Eea
Therefore,
\Bea
\E\left[\left|\g_{n}-\E_{1}[\g_{n}]\right|^{2}\right]&\leq&C_{1}\f{2(1+C_{0})}{b_{n}|\Im z|^{4}}|\E_{1}[A]|^{-2}+\f{C_{2}}{b_{n}}\left(\s^{2}+\f{2+2\mu_{4}}{|\Im z|^{2}}\right)\f{|\E_{1}[A]|^{-2}}{|\Im z|^{2}}\\
&\leq&\f{C}{b_{n}}\left(\f{1}{|\Im z|^{2}}+\f{1}{|\Im z|^{4}}\right)|\E_{1}[A]|^{-2},
\Eea
for some $C_{1},C_{2},C>0$ not depending on $z,n$. This implies
\Bea
\text{Var}(\g_{n})\leq\f{Cn}{b_{n}}\left(\f{1}{|\Im z|^{2}}+\f{1}{|\Im z|^{4}}\right)\left(\max\left\{y,|x|-\f{2}{y}\right\}\right)^{-2}.
\Eea
This completes the proof of proposition \ref{prop: bound on Var(gamma)}. \end{proof}

Now we proceed to the proofs of the asymptotic estimates. All the asymptotic estimates listed in Lemma \ref{main lemma without poincare condition} and Lemma \ref{main lemma with poincare condition} hold uniformly in the set $\{z\in\C:|\Im z|\geq \eta\}$ for any given $\eta>0$.
\begin{lem}\label{main lemma without poincare condition}
Let $M$ be an $n\times n$ symmetric band matrix as defined in \eqref{defn: random band matrix} which satisfies \eqref{eqn: moment condition} and $\E[|w_{ij}|^{8}]$ is uniformly bounded. Then    \begin{enumerate}[(i)]
\item\begin{equation}\label{difference between the traces of G and G^1}
G_{ii}^{(1)}-G_{ii}=\f{1}{A(z)}\left(G^{(1)}m^{(1)}\right)_{i}^{2}=\f{1}{A(z)}\left(\frac{1}{\sqrt{b_{n}}}\sum_{j\in I_{1}}G_{ij}^{(1)}w_{1j}\right)^{2}
\end{equation}
where $2\leq i\leq n$, $A(z)$, $m^{(1)}$ and $G^{(1)}$ are as defined in \eqref{definition of A}, \eqref{definition of m1} and \eqref{definition of G1}.
\item $\displaystyle{\left|\E\left[G_{ii}^{(1)}(z)\right]-\E[G_{ii}(z)]\right|=O\left(\f{1}{b_{n}}\right).}$
\item 
\begin{equation}\label{eqn: estimate of off-diagonal entries}
\E[|G_{12}|^{2}]=O\left(\f{1}{b_{n}}\right)\f{1}{|\Im z|^{6}},\
\E[|G_{12}|^{4}]=O\left(\f{1}{b_{n}^{2}}\right)\f{1}{|\Im z|^{12}}\ \mbox{and}\ \
\E[|G_{12}|^{8}]=O\left(\f{1}{b_{n}^{4}}\right)\f{1}{|\Im z|^{24}}.
\end{equation}
\item Let us denote the averaging with respect to $\{w_{1i}\}_{1\leq i\leq n}$ by $\E_1$. Then
\begin{equation}\label{eqn: b_nA(z_1)A(z_2)}
b_{n}\E_{1}\left[A^{\circ}(z_{1})A^{\circ}(z_{2})\right]=\s^{2}+\f{2}{b_{n}}\sum_{i,j\in I_{1}}G_{ij}^{(1)}(z_{1})G_{ij}^{(1)}(z_{2})+\f{\kappa_{4}}{b_{n}}\sum_{i\in I_{1}}G_{ii}^{(1)}(z_{1})G_{ii}^{(1)}(z_{2})+\f{1}{b_{n}}\widetilde{\g_{n-1}}(z_{1})\widetilde{\g_{n-1}}(z_{2})\end{equation}
where
$
\widetilde{\g_{n-1}}(z)=\sum_{i\in I_{1}}\left(G_{ii}^{(1)}-\E[G_{ii}^{(1)}(z)]\right)\ \
\text{and}\;\; I_{1}=\{1<i\leq n:(1,i)\in I_n\}.
$
\item \begin{equation}\label{A(z_1)B(z_2)}
\E_{1}\left[A^{\circ}(z_{1})B^{\circ}(z_{2})\right]=\f{d}{dz_{2}}\E_{1}\left[A^{\circ}(z_{1})A^{\circ}(z_{2})\right] \  \mbox{where}\  \displaystyle{B(z_{2})=\left\langle G^{(1)}(z_{2})G^{(1)}(z_{2})m^{(1)},m^{(1)}\right\rangle}.
\end{equation}
\end{enumerate}
\end{lem}
\begin{lem} \label{main lemma with poincare condition}
Let $M$ be an $n\times n$ symmetric band matrix as defined in \eqref{defn: random band matrix} which satisfies \eqref{eqn: moment condition}. Also assume that the probability distribution of $w_{jk}$ satisfies the Poincar\'e inequality with some uniform constant $m$ which does not depend on $n,j,k$. Then
\begin{enumerate}[(i)]
\item \begin{equation}\label{eqn: variance of sum of Gii is small}
\text{Var}\Big(\sum_{(1,i)\in I_{n}}G_{ii}\Big)=O(1)\ \mbox{ and } \ \text{Var}(G_{11}(z))=O\left(\f{1}{b_{n}}\right).\end{equation}
\item
\bea
\E\left[\left|A^{\circ}\right|^{4}\right]=O\left(\f{1}{b_{n}^{2}}\right), & &\E\left[\left|A^{\circ}\right|^{3}\right]=O\left(\f{1}{b_{n}^{3/2}}\right)\label{eqn: estimate of A^4}\\
\E\left[\left|B^{\circ}\right|^{4}\right]&=&O\left(\f{1}{b_{n}^{2}}\right)\label{eqn: estimate of B^4}
\eea
\item \begin{equation}\label{eqn: bound on variance of b_nA(z_1)A(z_2)}
\text{Var}\left\{b_{n}\E_{1}\left[A^{\circ}(z_{1})A^{\circ}(z_{2})\right]\right\}=O\left(\f{1}{b_{n}}\right) \mbox{ and } \text{Var}\left\{b_{n}\E_{1}\left[A^{\circ}(z_{1})B^{\circ}(z_{2})\right]\right\}=O\left(\f{1}{b_{n}}\right)
\end{equation}
\item  \bea
\E\left[\left|{\g_{n-1}^{\circ}}(z)-\g_{n}^{\circ}(z)\right|^{4}\right]=O\left(\f{1}{b_{n}^{2}}\right) \ \mbox{ and }\ \E\left[\left|\g_{n}^{\circ}\right|^{4}\right]=O\left(\frac{n^{2}}{b_{n}^{2}}\right).\label{eqn: fourth moment of g_n^1-g_n}
\eea
\item \begin{equation}\label{eqn: asymptotic of f_n(z)}
\frac1n \E \left[Tr G(z)\right]=f(z)+O\left(\frac{1}{|\Im z|^6b_n}\right)
\mbox{ where } f(z)=\f{1}{4}\left(-z+\sqrt{z^{2}-8}\right).
\end{equation}
\item \begin{equation}\label{eqn: limit of A(z),B(z)}
(\E[A(z)])^{-1}=- f(z)+O(b_n^{-1}) \ \mbox{ and }\ \E[B(z)]= 2f'(z)+O(b_n^{-1}).
\end{equation}

\end{enumerate}
\end{lem}

\begin{proof}[Proof of Lemma \ref{main lemma without poincare condition}:]
\noindent\textbf{Proof of (i):}  
Suppose $(X_1,X_2,\ldots,X_n)$ is  a $n$ dimensional normal random vector with a positive definite covariance matrix $A^{-1}$ and a mean $A^{-1}\underline{h}$, where $\underline{h}\in \mathbb R^n$. Then we have
\bea
\int\exp\left[-\frac12 \langle A\underline{x},\underline{x}\rangle +\langle \underline{h},\underline{x}\rangle\right]\;d\underline{x}&=&(2\pi)^{n/2}|\det A|^{-1/2}\exp\left[\frac 12\langle A^{-1}\underline{h},\underline{h}\rangle \right],\label{eqn: first Gaussian integration}\\
\f{\int x_{i}x_{j}\exp\left[-\frac 12 \langle A\underline{x},\underline{x}\rangle +\langle \underline{h},\underline{x}\rangle\right]\;d\underline{x}}{\int\exp\left[-\frac 12\langle A\underline{x},\underline{x}\rangle +\langle \underline{h},\underline{x}\rangle \right]\;d\underline{x}}&=&(A^{-1})_{ij}+(A^{-1}\underline{h})_{i}(A^{-1}\underline{h})_{j}.\label{eqn: Second Gaussian integration}
\eea
where $\underline{x}=(x_1,x_2,\ldots,x_n)^{T}$. In particular, for $\underline{h}=0$,
\bea\label{eqn: third Gaussian integration}
(A^{-1})_{ij}=\f{\int x_{i}x_{j}\exp[-\frac 12\langle A\underline{x},\underline{x}\rangle]\;d\underline{x}}{\int\exp[-\frac 12\langle A\underline{x},\underline{x}\rangle ]\;d\underline{x}}.
\eea 
Now doing the integrations in \eqref{eqn: third Gaussian integration}  with respect to all variables except  $x_1$, and using \eqref{eqn: first Gaussian integration} we get
\Bea\int\exp[-\frac 12\langle A\underline{x},\underline{x}\rangle ]\;d\underline{x}&=&\int\exp[-\frac{a_{11}x_1^2}{2}]\int \exp[-\frac 12 \langle A_1\underline{x}^{(1)},\underline{x}^{(1)}\rangle -\langle x_1\underline{a_1},\underline{x}^{(1)}\rangle  ]\ d\underline{x}\\
&=&\frac{(2\pi)^{\frac{n-1}{2}}}{|\det A_1|^{1/2}}\int \exp[-\frac{x_1^2}{2}(a_{11}-\langle A_1^{-1}\underline{a_1},\underline{a_1}\rangle)]\ dx_1
\Eea
where $\underline{x}^{(1)}=(x_2,x_3,\ldots,x_n)^T$,  $\underline{a_1}=(a_{12},a_{13},\ldots,a_{1n})^T$ and $A_1=((A_{1})_{ij})_{i,j=2}^{n}$ is the $(n-1)\times (n-1)$ matrix obtained from $A$ after removing first row and first column, and for $i,j\neq 1$, using \eqref{eqn: Second Gaussian integration} and \eqref{eqn: first Gaussian integration} we get
\Bea
&&\int x_{i}x_{j}\exp[-\frac 12\langle A\underline{x},\underline{x}\rangle]\;d\underline{x}\\
&=&\int\exp[-\frac{a_{11}x_1^2}{2}]\int x_ix_j \exp[-\frac 12 \langle A_1\underline{x}^{(1)},\underline{x}^{(1)}\rangle -\langle x_1\underline{a_1},\underline{x}^{(1)}\rangle  ]\ d\underline{x}^{(1)}dx_1\\
&=&\int\exp[-\frac{a_{11}x_1^2}{2}]\  [(A_1^{-1})_{ij}+x_1^2 ( A_1^{-1}\underline{a_1})_i (A_1^{-1}\underline{a_1})_j]\int \exp[-\frac 12 \langle A_1\underline{x}^{(1)},\underline{x}^{(1)}\rangle -\langle x_1\underline{a_1},\underline{x}^{(1)}\rangle  ]\ d\underline{x}^{(1)}dx_1\\
&=&\frac{(2\pi)^{\frac{n-1}{2}}}{|\det A_1|^{1/2}}\int [(A_1^{-1})_{ij}+x_1^2 ( A_1^{-1}\underline{a_1})_i (A_1^{-1}\underline{a_1})_j]\ \exp[-\frac{x_1^2}{2}(a_{11}-\langle A_1^{-1}\underline{a_1},\underline{a_1}\rangle)]\ dx_1.
\Eea
Therefore, from  \eqref{eqn: third Gaussian integration} we get
\Bea
(A^{-1})_{ij}&=&(A_1^{-1})_{ij}+( A_1^{-1}\underline{a_1})_i (A_1^{-1}\underline{a_1})_j \ \frac{\int x_1^2 \exp[-\frac{x_1^2}{2}(a_{11}-\langle A_1^{-1}\underline{a_1},\underline{a_1}\rangle)]\ dx_1}{\int \exp[-\frac{x_1^2}{2}(a_{11}-\langle A_1^{-1}\underline{a_1},\underline{a_1}\rangle)]\ dx_1}\\\nonumber
&=&(A_1^{-1})_{ij}+\frac{( A_1^{-1}\underline{a_1})_i (A_1^{-1}\underline{a_1})_j}{a_{11}-\langle A_1^{-1}\underline{a_1},\underline{a_1}\rangle}.
\Eea
Applying the above formula for $A=(M-zI)$, where  $z\in\R$, $|z|>\|M\|$, we obtain
$$G_{ij}=G^{(1)}_{ij}+ \frac{( G^{(1)} m^{(1)})_i (G^{(1)}m^{(1)})_j}{\frac{w_{11}}{\sqrt{b_n}}-z-\langle G^{(1)}m^{(1)},m^{(1)}\rangle},\;\;\;i,j\geq 2,$$
where $m^{(1)}$, $G^{(1)}$ are  as defined in \eqref{definition of m1}, \eqref{definition of G1} respectively. From the above formula we  obtain
\Bea
G_{ii}-G_{ii}^{(1)}=-\f{(G^{(1)}m^{(1)})_{i}^{2}}{A(z)},\;\;\;\mbox{for all}\;\  2\leq i\leq n,
\Eea
where $A(z)$ is as defined in \eqref{definition of A}. The above is true for all $z\in\R$ such that $|z|>\|M\|$. By analytic continuity one can extend it to the whole complex plane. This completes the proof.\\

\noindent
\textbf{Proof of (ii):} Recall $I_{1}=\{1<i\leq n:(1,i)\in I_n\}$. Now using \eqref{difference between the traces of G and G^1} and \eqref{eqn: norm of G is bounded by Imz} we have
\Bea
\left|\E\left[G_{ii}^{(1)}(z)\right]-\E[G_{ii}(z)]\right|&=&\left|\E\left[\f{1}{A}\left(G^{(1)}m^{(1)}\right)_{i}^{2}\right]\right|\\
&=&\left|\E\left[\f{1}{A}\left(\f{1}{\sqrt{b_{n}}}\sum_{j\in I_{1}}G_{ij}^{(1)}w_{1j}\right)^{2}\right]\right|\\
&\leq&\f{1}{b_{n}}\f{1}{|\Im z|}\E\left[\left|\sum_{j\in I_{1}}G_{ij}^{(1)}w_{1j}\right|^{2}\right]\\
&\leq&\f{1}{b_{n}|\Im z|}\E\left[\sum_{j\in I_{1}}|G_{ij}^{(1)}|^{2}w_{1j}^{2}+\sum_{j_{1}\neq j_{2}\in I_1}G_{ij_{1}}^{(1)}\overline{ G_{ij_{2}}^{(1)}}w_{1j_{1}}w_{1j_{2}}\right]\\
&=&\f{1}{b_{n}|\Im z|}\E\E_{1}\left[\sum_{j\in I_{1}}|G_{ij}^{(1)}|^{2}w_{1j}^{2}+\sum_{j_{1}\neq j_{2}\in I_1}G_{ij_{1}}^{(1)}\overline{G_{ij_{2}}^{(1)}}w_{1j_{1}}w_{1j_{2}}\right]\\
&=&\f{1}{b_{n}|\Im z|}\E\left[\sum_{j\in I_{1}}|G_{1j}^{(1)}|^{2}\right]\\
&\leq&\f{1}{b_{n}|\Im z|}\E\|G^{(1)}\|^{2}\leq\f{1}{b_{n}|\Im z|^{3}}.
\Eea

\noindent
\textbf{Proof of (iii):}
Using the resolvent formula given in ~\cite{erdHos2011universality}, we have
\Bea
G_{12}=-G_{22}G_{11}^{(2)}K_{12}^{(12)},
\Eea
where $G^{(2)}$ is the resolvent of the $(n-1)\times(n-1)$ minor obtained by removing the $k$th row and $k$th column from the matrix $M$, $K_{12}^{(12)}=m_{12}-m_{(1)}G^{(12)}m_{(2)}$, $m_{(1)}=\f{1}{\sqrt{b_{n}}}(w_{13},w_{14},\ldots,w_{1n})$, $m_{(2)}=\f{1}{\sqrt{b_{n}}}(w_{23},w_{24},\ldots,w_{2n})^T$, $G^{(ij)}=\left(M^{(ij)}-zI\right)^{-1}$, and $M^{(ij)}$ is $(n-2)\times (n-2)$ matrix obtained from $M$ after removing  $i$th and $j$th rows and columns. Therefore,
\Bea
\E[|G_{12}|^{2}]&=&\E\left[\left|G_{22}G_{11}^{(2)}K_{12}^{(12)}\right|^{2}\right]\\
&\leq&\f{1}{|\Im z|^{2}}\f{1}{|\Im z|^{2}}\E\left[\left|m_{12}-m_{(1)}G^{(12)}m_{(2)}\right|^{2}\right]\\
&=&\f{1}{|\Im z|^{4}}\E\left[\left|\f{w_{12}}{\sqrt{b_{n}}}-\f{1}{b_{n}}\sum_{\stackrel{(1,i),(2,j)\in I_{n}}{i,j\neq 1,2}}G_{ij}^{(12)}w_{1i}w_{2j}\right|^{2}\right]\\
&\leq&\f{1}{|\Im z|^{4}}\E\E_{\leq 2}\left[\f{w_{12}^{2}}{b_{n}}+\f{1}{b_{n}^{2}}\sum_{\stackrel{(1,i), (2,j)\in I_{n}}{i,j\neq 1,2}}|G_{ij}^{(12)}|^{2}w_{1i}^{2}w_{2j}^{2}\right]\\
&\leq&\f{1}{|\Im z|^{4}}\E\left[\f{1}{b_{n}}+\f{1}{b_{n}^{2}}\sum_{i,j}|G_{ij}^{(12)}|^{2}\E_{\leq 2}[w_{1i}^{2}]\E_{\leq 2}[w_{2j}^{2}]\right]\\
&\leq&\f{1}{|\Im z|^{4}}\E\left[\f{1}{b_{n}}+\f{1}{b_{n}^{2}}\f{b_{n}}{|\Im z|^{2}}\right]=O\left(\f{1}{b_{n}}\right)\f{1}{|\Im z|^{6}},
\Eea
where $E_{\leq 2}$ is the averaging with respect to the first two rows and columns. Similarly, we can prove that $\E[|G_{12}|^{4}]=O\left(\f{1}{b_{n}^{2}}\right)\f{1}{|\Im z|^{12}}$, and $\E[|G_{12}|^{8}]=O\left(\f{1}{b_{n}^{4}}\right)\f{1}{|\Im z|^{24}}$.\\

\noindent\textbf{Proof of (iv):}
We know that
\Bea
A(z_{1})&=&z_{1}-\f{w_{11}}{\sqrt{b_{n}}}+\left\langle G^{(1)}m^{(1)},m^{(1)}\right\rangle,\\
\text{and }\;\;A^{\circ}(z_{1})&=&-\f{w_{11}}{\sqrt{b_{n}}}+\f{1}{b_{n}}\sum_{i\neq j\in I_{1}}G_{ij}^{(1)}w_{1i}w_{1j}+\f{1}{b_{n}}\sum_{i\in I_{1}}G_{ii}^{(1)}w_{1i}^{2}-\f{1}{b_{n}}\sum_{i\in I_{1}}\E[G_{ii}^{(1)}].
\Eea
Now we can estimate
\Bea
&&b_{n}\E_{1}\left[A^{\circ}(z_{1})A^{\circ}(z_{2})\right]\\
&=&\s^{2}+\f{1}{b_{n}}\E_{1}\left[\sum_{\stackrel{i_{1}\neq j_{1}\in I_{1}}{i_{2}\neq j_{2}\in I_{1}}}G_{i_{1}j_{1}}^{(1)}(z_{1})w_{1i_{1}}w_{1j_{1}}G_{i_{2}j_{2}}^{(1)}(z_{2})w_{1i_{2}}w_{1j_{2}}\right]+\f{1}{b_{n}}\E_{1}\left[\sum_{i,j\in I_{1}}G_{ii}^{(1)}(z_{1})G_{jj}^{(1)}(z_{2})w_{1i}^{2}w_{1j}^{2}\right]\\
&&-\f{1}{b_{n}}\E\left[\sum_{i\in I_{1}}G_{ii}^{(1)}(z_{2})\right]\E_{1}\left[\sum_{i\in I_{1}}G_{ii}^{(1)}(z_{1})w_{1i}^{2}\right]-\f{1}{b_{n}}\E\left[\sum_{i\in I_{1}}G_{ii}^{(1)}(z_{1})\right]\E_{1}\left[\sum_{i\in I_{1}}G_{ii}^{(1)}(z_{2})w_{1i}^{2}\right]\\
&&+\f{1}{b_{n}}\E\left[\sum_{i\in I_{1}}G_{ii}^{(1)}(z_{1})\right]\E\left[\sum_{i\in I_{1}}G_{ii}^{(1)}(z_{2})\right]\\
&=&\s^{2}+\f{2}{b_{n}}\sum_{i\neq j\in I_{1}}G_{ij}^{(1)}(z_{1})G_{ij}^{(1)}(z_{2})+\f{1}{b_{n}}\sum_{i\neq j\in I_{1}}G_{ii}^{(1)}(z_{1})G_{jj}^{(1)}(z_{2})
+\f{\mu_{4}}{b_{n}}\sum_{i\in I_{1}}G_{ii}^{(1)}(z_{1})G_{ii}^{(1)}(z_{2})\\
&&+\f{1}{b_{n}}\widetilde{\g_{n-1}}(z_{1})\widetilde{\g_{n-1}}(z_{2})-\f{1}{b_{n}}\left(\sum_{i\in I_{1}}G_{ii}^{(1)}(z_{1})\right)\left(\sum_{i\in I_{1}}G_{ii}^{(1)}(z_{2})\right)\\
&=&\s^{2}+\f{2}{b_{n}}\sum_{i, j\in I_{1}}G_{ij}^{(1)}(z_{1})G_{ij}^{(1)}(z_{2})+\f{\mu_{4}}{b_{n}}\sum_{i\in I_{1}}G_{ii}^{(1)}(z_{1})G_{ii}^{(1)}(z_{2})-\f{3}{b_{n}}\sum_{i\in I_{1}}G_{ii}^{(1)}(z_{1})G_{ii}^{(1)}(z_{2})\\
&&+\f{1}{b_{n}}\widetilde{\g_{n-1}}(z_{1})\widetilde{\g_{n-1}}(z_{2})\\
&=&\s^{2}+\f{2}{b_{n}}\sum_{i, j\in I_{1}}G_{ij}^{(1)}(z_{1})G_{ij}^{(1)}(z_{2})+\f{\kappa_{4}}{b_{n}}\sum_{i\in I_{1}}G_{ii}^{(1)}(z_{1})G_{ii}^{(1)}(z_{2})+\f{1}{b_{n}}\widetilde{\g_{n-1}}(z_{1})\widetilde{\g_{n-1}}(z_{2}),
\Eea
where $\kappa_{4}=\mu_{4}-3.$\\

\noindent\textbf{Proof of (v):}
Observe that
\Bea
B(z_{2})=\left\langle G^{(1)}G^{(1)}m^{(1)},m^{(1)}\right\rangle
=\f{1}{b_{n}}\sum_{i,j\in I_{1}}\left(G^{(1)}G^{(1)}\right)_{ij}w_{1i}w_{1j}
=\f{1}{b_{n}}\sum_{i,j\in I_{1}}\sum_{k=2}^{n}G_{ik}^{(1)}G_{kj}^{(1)}w_{1i}w_{1j},
\Eea
and
\Bea
\f{d}{dz_{2}}G_{ij}^{(1)}(z_{2})=\left( G^{(1)}(z_{2})G^{(1)}(z_{2})\right)_{ij}=\sum_{k=2}^{n}G_{ik}^{(1)}(z_{2})G_{kj}^{(1)}(z_{2}).
\Eea
Now,
proceed as in \textbf{(iv)} and use the above facts to prove the result. Here we skip the details.

\end{proof}

\begin{proof}[Proof of Lemma \ref{main lemma with poincare condition}:]
\textbf{Proof of (i):} Since $w_{jk}$ satisfies the Poincar\'e inequality with constant $m$ and the Poincar\'e  inequality tensorises, the joint distribution of $\{w_{jk}\}_{(j,k)\in I_n^+}$ on  $\mathbb R^{n(b_n+1)} $ satisfies the Poincar\'e inequality with same constant $m$. Therefore we have
\Bea
\text{Var}\left(\Phi\left(\{w_{jk}\}_{(j,k)\in I_{n}^{+}}\right)\right)\leq\f{1}{m}\sum_{(j,k)\in I_{n}^{+}}\E\left[\left|\f{\partial\Phi}{\partial w_{jk}}\right|^{2}\right],
\Eea
for any continuously differentiable function $\Phi$. Therefore,

\bea\label{eqn: estimation of gradient of gamma tilde n-1}
\text{Var}\left(\sum_{(1,i)\in I_{n}}G_{ii}\right)&\leq&\f{1}{m}\sum_{(j,k)\in I_{n}^{+}}\E\left[\left|\f{\partial}{\partial w_{jk}}\sum_{(1,i)\in I_{n}}G_{ii}\right|^{2}\right]\\
&\leq&\f{4}{mb_{n}}\sum_{(j,k)\in I_{n}^{+}}\E\left[\left|\sum_{(1,i)\in I_{n}}G_{ij}G_{ki}\right|^{2}\right]\nonumber\\
&=&\f{4}{mb_{n}}\sum_{(j,k)\in I_{n}^{+}}\E\left[|\alpha_{kj}|^{2}\right]\;\;\;\text{where $\alpha_{kj}=\sum_{(1,i)\in I_{n}}G_{ki}G_{ij}$}\nonumber\\
&\leq&\f{4}{mb_{n}}\sum_{j,k=1}^{n}\E\left[|\alpha_{kj}|^{2}\right]\nonumber\\
&=&\f{4}{mb_{n}}\E\left[\|VV^{T}\|_{Fb}^{2}\right]\nonumber\\
&=&\f{4}{mb_{n}}\E\left[\sum_{i=1}^{n}|\beta_{i}|^{2}\right],\nonumber
\eea
where
$$
V=\left[\begin{array}{lllllll}
G_{11}&G_{12}&\cdots&G_{1k_n}&0&\cdots&0\\
G_{21}&G_{22}&\cdots&G_{2k_n}&0&\cdots&0\\
& &&\vdots&&&\\
G_{n1}&G_{n2}&\cdots&G_{nk_n}&0&\cdots&0
\end{array}
\right]_{n\times n}
$$

and $\|\cdot\|_{Fb}$ stands for the Frobenius norm, and $\beta_{i}$s are the eigenvalues of $VV^{T}$. Here, we denote the set $\{i:(1,i)\in I_n\}$ by $\{1,2,\ldots,k_n\}$. Observe that $k_n=2b_n+1$. Since $rank(VV^{T})\leq k_{n}=O(b_n)$, we have $\#\{i:\beta_{i}\neq 0\}\leq k_n=O(b_{n})$. Also we know that $\|V\|\leq \|G\|$. Therefore,
\Bea
|\beta_{i}|^{2}\leq\|VV^{T}\|^{2}\leq\|G\|^{4}\leq\f{1}{|\Im z|^{4}}.
\Eea
Consequently, we have
\bea\label{final result var sum G_{ii}}
\text{Var}\left(\sum_{(1,i)\in I_{n}}G_{ii}\right)\leq\f{4}{mb_{n}}\E\left[\sum_{i=1}^{n}|\beta_{i}|^{2}\right]\leq \f{4}{mb_{n}}\f{O(b_{n})}{|\Im z|^{4}}=O(1).
\eea
This completes proof of first part of \eqref{eqn: variance of sum of Gii is small}.

Recall the definition of $A$ from \eqref{definition of A}, $A=z-\f{1}{\sqrt{b_{n}}}w_{11}+\left(G^{(1)}m^{(1)},m^{(1)}\right)$. Then
\Bea
A^{\circ}&=&A-\E[A]\\
&=&-\f{1}{\sqrt{b_{n}}}w_{11}+\f{1}{b_{n}}\sum_{\stackrel{i\neq j}{i,j\in I_{1}}}G_{ij}^{(1)}w_{1i}w_{1j}+\f{1}{b_{n}}\sum_{i\in I_1}\left(G_{ii}^{(1)}w_{1i}^{2}-\E[G_{ii}^{(1)}]\right),
\Eea
Consider
\bea\label{eqn: A_1}
A_{1}^{\circ}&=&A-\E_{1}[A] \nonumber \\
&=&-\f{1}{\sqrt{b_{n}}}w_{11}+\f{1}{b_{n}}\sum_{\stackrel{i\neq j}{i,j\in I_{1}}}G_{ij}^{(1)}w_{1i}w_{1j}+\f{1}{b_{n}}\sum_{i\in I_1}\left(G_{ii}^{(1)}w_{1i}^{2}-G_{ii}^{(1)}\right).
\eea
So we have
\bea\label{eqn: relation A and A_1}
A^{\circ}-A_{1}^{\circ}=\f{1}{b_{n}}\sum_{i\in I_1}\left(G_{ii}^{(1)}-\E\left[G_{ii}^{(1)}\right]\right) =: \f{1}{b_{n}}\widetilde{\g_{n-1}}.
\eea
Hence
\Bea
\E[|A^{\circ}|^{2}]=\E[|A_{1}^{\circ}+b_{n}^{-1}\widetilde{\g_{n-1}}|^{2}]
\leq 2\left[\E[|A_{1}^{\circ}|^{2}]+\f{1}{b_{n}^{2}}\E[|\widetilde{\g_{n-1}}|^{2}]\right].
\Eea
From \eqref{estimate of A1 as O(1/b)}, we know that $\E[|A_{1}^{\circ}|^{2}]=O\left(\f{1}{b_{n}}\right)$ and from \eqref{final result var sum G_{ii}}, We have $\E[|\widetilde{\g_{n-1}}|^{2}]=O(1)$. Combining these two facts and using \eqref{eqn: schur complement formula}, we have
$$\text{Var}(G_{11}(z))=\E\left|\f{1}{A}-\E\f{1}{A}\right|^{2}\leq\E\left|\f{1}{A}-\f{1}{\E A}\right|^{2}=\E\left|\f{A^{\circ}}{A\E A}\right|^{2}=O\left(\f{1}{b_{n}}\right).$$
This completes the proof of second part.\\

\noindent
\textbf{Proof of (ii):}
\textit{Proof of \eqref{eqn: estimate of A^4}:} Recall from \eqref{eqn: A_1}
\bea
A_{1}^{\circ}=-\f{w_{11}}{\sqrt{b_{n}}}+\f{1}{b_{n}}\sum_{i\neq j\in I_{1}}G_{ij}^{(1)}w_{1i}w_{1j}+\f{1}{b_{n}}\sum_{i\in I_{1}}G_{ii}^{(1)}(w_{1i}^{2})^{\circ}=: T_{1}+T_{2}+T_{3}.\label{eqn: A1 as a sum of T1, T2 and T3}
\eea
We have $\E[|T_{1}|^{4}]=O\left(\f{1}{b_{n}^{2}}\right)$. Now
\Bea
\E\left[\left|T_{2}\right|^{4}\right]=\f{1}{b_{n}^{4}}\E\left[\sum_{i\neq j, k\neq l, p\neq q, s\neq t\in I_{1}}G_{ij}^{(1)}\overline{G_{kl}^{(1)}}G_{pq}^{(1)}\overline{G_{st}^{(1)}}w_{1i}w_{1j}w_{1k}w_{1l}w_{1p}w_{1q}w_{1s}w_{1t}\right].
\Eea

We use the similar technique as the moment method in the proof of the Semicircle Law. In the above sum of expectations, we have nonzero terms if the indices of $w_{1m}$'s match in a certain way. Non zero contribution to $\E[|T_{2}|^{4}]$ come from the two types of matches.
\begin{figure}[h!]
\centering
\includegraphics[scale=0.6]{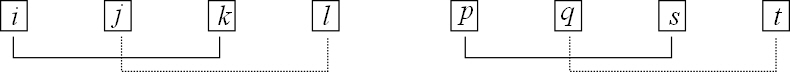}
\caption{Type I matching}
\bigskip
\includegraphics[scale=0.6]{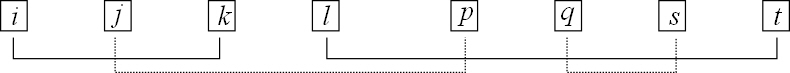}
\caption{Type II matching}
\end{figure}

\noindent
\textbf{Type I:} Contribution from this kind of matching is
\Bea
\f{1}{b_{n}^{4}}\E\left[\sum_{i\neq j,p\neq k\;\in I_{1}}|G_{ij}^{(1)}|^{2}|G_{pq}^{(1)}|^{2}w_{1i}^{2}w_{1j}^{2}w_{1p}^{2}w_{1q}^{2}\right]
&=&\f{1}{b_{n}^{4}}\E\E_{1}\left[\sum_{i\neq j, p\neq k\;\in I_{1}}|G_{ij}^{(1)}|^{2}|G_{pq}^{(1)}|^{2}w_{1i}^{2}w_{1j}^{2}w_{1p}^{2}w_{1q}^{2}\right]\\
&=&\f{1}{b_{n}^{4}}\sum_{i\neq j}\sum_{p\neq q}\E\left[|G_{ij}^{(1)}|^{2}|G_{pq}^{(1)}|^{2}\right]\\
&\leq&\f{1}{b_{n}^{4}}\sum_{i\neq j}\sum_{p\neq q}\sqrt{\E\left[\left|G_{ij}^{(1)}\right|^{4}\right]\E\left[\left|G_{pq}^{(1)}\right|^{4}\right]}\\
&=&\f{1}{b_{n}^{4}}\sum_{i\neq j}\sum_{p\neq q}O\left(\f{1}{b_{n}^{2}}\right)\ \ (\mbox{using \eqref{eqn: estimate of off-diagonal entries}})\\
&=&O\left(\f{1}{b_{n}^{2}}\right).
\Eea
\textbf{Type II:} Similarly, contribution from the type II matching is
\Bea
\f{1}{b_{n}^{4}}\E\left[\sum_{i\neq j}\sum_{q\neq l\in I_{1}}G_{ij}^{(1)}\overline{G_{il}^{(1)}}G_{jq}^{(1)}\overline{G_{ql}^{(1)}}w_{1i}^{2}w_{1j}^{2}w_{1q}^{2}w_{1l}^{2}\right]=O\left(\f{1}{b_{n}^{2}}\right).
\Eea
Similarly, $\displaystyle{\E[|T_{3}|^{4}]}=O\left(\f{1}{b_{n}^{2}}\right)$. Hence
\bea
\E[|A_{1}^{\circ}|^{4}]=O\left(\f{1}{b_{n}^{2}}\right).\label{eqn: estimate of A_1^4}
\eea
Using Lemma 4.4.3. from \cite{anderson2010introduction} with the help of the Poincar\'e inequality, we have $\E\left[\left|\widetilde{\g_{n-1}}\right|^{4}\right]\leq C\|\|\nabla\widetilde{\g_{n-1}}\|_{2}\|_{\infty}^{4}$, where $C$ is a constant depends only on the constant $m$ of the Poincar\'e inequality. Following the arguments given at the  right side of \eqref{eqn: estimation of gradient of gamma tilde n-1} onward and \eqref{final result var sum G_{ii}}, one can show  that $\|\nabla\widetilde{\g_{n-1}}\|_{2}\leq \f{C}{|\Im z|^{4}}$, where $C$ depends only on $m$. Hence $\E\left[\left|\widetilde{\g_{n-1}}\right|^{4}\right]=O(1)$. Consequently, using  relation \eqref{eqn: relation A and A_1} and \eqref{eqn: estimate of A_1^4}, we have  $\E[|A^{\circ}|^{4}]=O\left(\f{1}{b_{n}^{2}}\right)$. Then $\E[|A^{\circ}|^{3}]\leq \left(\E[|A^{\circ}|^{4}]\right)^{3/4}=O\left(\f{1}{b_{n}^{3/2}}\right).$\\\\

%
%

\noindent\textit{Proof of \eqref{eqn: estimate of B^4}:} First we write $B$ as
\Bea
B=\left\langle G^{(1)}G^{(1)}m^{(1)},m^{(1)}\right\rangle=\left\langle H^{(1)}m^{(1)},m^{(1)}\right\rangle =\f{1}{b_{n}}\sum_{i,j\in I_{1}}H_{ij}^{(1)}w_{1i}w_{1j},
\Eea
where $H^{(1)}=G^{(1)}G^{(1)}$. Define
\Bea
B_{1}^{\circ}:=\f{1}{b_{n}}\sum_{i\neq j\in I_{1}}H_{ij}^{(1)}w_{1i}w_{1j}+\f{1}{b_{n}}\sum_{i\in I_{1}}H_{ii}^{(1)}(w_{1i}^{2})^{\circ}.
\Eea
Then we can write
\Bea
B^{\circ}&=&B-\E[B]\\
&=&\f{1}{b_{n}}\sum_{i\neq j\in I_{1}}H_{ij}w_{1i}w_{1j}+\f{1}{b_{n}}\sum_{i\in I_{1}}\left[H_{ii}^{(1)}w_{1i}^{2}-\E[H_{ii}^{(1)}]\right]\\
&=&B_{1}^{\circ}+\f{1}{b_{n}}\sum_{i\in I_{1}}\left(H_{ii}^{(1)}-\E[H_{ii}^{(1)}]\right)\\
&=&B_{1}^{\circ}+\f{1}{b_{n}}\overline{\g_{n-1}},
\Eea
where
\Bea
\overline{\g_{n-1}}(z)=\sum_{i\in I_{1}}\left(H_{ii}^{(1)}-\E[H_{ii}^{(1)}]\right)=\sum_{i\in I_{1}}\sum_{j=2}^{n}\left(G_{ij}^{(1)}G_{ji}^{(1)}-\E\left[G_{ij}^{(1)}G_{ji}^{(1)}\right]\right)=\f{d}{dz}\widetilde{\g_{n-1}}(z).
\Eea

Proceeding as in the estimate of $\E[|A_{1}^{\circ}|^{4}]$, we can show
\bea
\E[|B_{1}^{\circ}|^{4}]=O\left(\f{1}{b_{n}^{2}}\right).\label{eqn: estimate of B_1^4}
\eea
We have shown that $\E[|\widetilde{\g_{n-1}}(z)|^{4}]=O(1)$. Using this fact and Cauchy's theorem we have $\E[|\overline{\g_{n-1}}(z)|^{4}]=O(1)$. Hence we have the result.\\




\noindent\textbf{Proof of (iii):}
\Bea
\text{Var}\left\{b_{n}\E_{1}\left[A^{\circ}(z_{1})A^{\circ}(z_{2})\right]\right\}=\text{Var}(T_{1})+\text{Var}(T_{2})+\text{Var}(T_{3})+2\text{Cov}(T_{1},T_{2})+2\text{Cov}(T_{2},T_{3})+2\text{Cov}(T_{3},T_{1}),
\Eea
where
\Bea
T_{1}=\f{2}{b_{n}}\sum_{i,j\in I_{1}}G_{ij}^{(1)}(z_{1})G_{ij}^{(1)}(z_{2}),\
T_{2}=\f{\kappa_{4}}{b_{n}}\sum_{i\in I_{1}}G_{ii}^{(1)}(z_{1})G_{ii}^{(1)}(z_{2})\ \mbox{and}\
T_{3}=\f{1}{b_{n}}\widetilde{\g_{n-1}}(z_{1})\widetilde{\g_{n-1}}(z_{2}).
\Eea
Now, Var$\displaystyle{(T_{2})=\f{\kappa_{4}^{2}}{b_{n}^{2}}\text{Var}\left\{\sum_{i\in I_{1}}G_{ii}^{(1)}(z_{1})G_{ii}^{(1)}(z_{2})\right\}}$ and
\Bea
\text{Var}\left\{G_{ii}^{(1)}(z_{1})G_{ii}^{(1)}(z_{2})\right\}&=&\E\left|G_{ii}^{(1)}(z_{1})G_{ii}^{(1)}(z_{2})-\E[G_{ii}^{(1)}(z_{1})G_{ii}^{(1)}(z_{2})]\right|^{2}\\
&\leq&\f{2}{|\Im z_{1}|^{2}}\text{Var}\left(G_{ii}^{(1)}(z_{2})\right)+\f{2}{|\Im z_{2}|^{2}}\text{Var}\left(G_{ii}^{(1)}(z_{1})\right)\\
&=&\left(\f{1}{|\Im z_{1}|^{2}}+\f{1}{|\Im z_{2}|^{2}}\right)O\left(\f{1}{b_{n}}\right).
\Eea
Therefore,
\Bea
\text{Var}(T_{2})\leq\f{\kappa_{4}^{2}}{b_{n}^{2}}\left(b_{n}O\left(\f{1}{b_{n}}\right)+b_{n}^{2}O\left(\f{1}{b_{n}}\right)\right)=O\left(\f{1}{b_{n}}\right).
\Eea
Now
\Bea
\text{Var}(T_{3})&\leq&\f{1}{b_{n}^{2}}\text{Var}\left(\widetilde{\g_{n-1}}(z_{1})\widetilde{\g_{n-1}}(z_{2})\right)\\
&\leq&\f{1}{b_{n}^{2}}\E\left[|\widetilde{\g_{n-1}}(z_{1})|^{2}|\widetilde{\g_{n-1}}(z_{2})|^{2}\right]\\
&\leq&\f{1}{b_{n}^{2}}\sqrt{\E\left[|\widetilde{\g_{n-1}}(z_{1})|^{4}\right]}\sqrt{\E\left[|\widetilde{\g_{n-1}}(z_{2})|^{4}\right]}\\
&=&\f{1}{b_{n}^{2}}O(1).
\Eea
Last equality holds, since $\E\left[|\widetilde{\g_{n-1}}(z_{1})|^{4}\right]=O(1)$.  And finally
\Bea
\text{Var}(T_{1})=\f{4}{b_{n}^{2}}\text{Var}\left(\sum_{i,j\in I_{1}}G_{ij}^{(1)}(z_{1})G_{ij}^{(1)}(z_{2})\right).
\Eea
Now using the Poincar\'e inequality
\Bea
&&\text{Var}\left(\sum_{i,j\in I_{1}}G_{ij}^{(1)}(z_{1})G_{ij}^{(1)}(z_{2})\right)\\
&\leq&\f{1}{m}\sum_{(s,t)\in I_{n}^{+}}\E\left[\left|\f{\partial}{\partial w_{st}}\sum_{i,j\in I_{1}}G_{ij}^{(1)}(z_{1})G_{ij}^{(1)}(z_{2})\right|^{2}\right]\\
&\leq&\f{1}{mb_{n}}\sum_{(s,t)\in I_{n}^{+}}\E\left[\left|\sum_{i,j\in I_{1}}G_{is}^{(1)}(z_{1})G_{tj}^{(1)}(z_{1})G_{ij}^{(1)}(z_{2})+G_{ij}^{(1)}(z_{1})G_{is}^{(1)}(z_{2})G_{tj}^{(1)}(z_{2})\right|^{2}\right]\\
&\leq&\f{2}{mb_{n}}\sum_{(s,t)\in I_{n}^{+}}\E\left[\left|\sum_{i,j\in I_{1}}G_{is}^{(1)}(z_{1})G_{tj}^{(1)}(z_{1})G_{ij}^{(1)}(z_{2})\right|^{2}\right]+\f{2}{mb_{n}}\sum_{(s,t)\in I_{n}^{+}}\E\left[\left|\sum_{i,j\in I_{1}}G_{ij}^{(1)}(z_{1})G_{is}^{(1)}(z_{2})G_{tj}^{(1)}(z_{2})\right|^{2}\right]\\
&=:&I_{1}+I_{2}.
\Eea
We estimate
\Bea
I_{1}&=&\f{2}{mb_{n}}\sum_{(s,t)\in I_{n}^{+}}\E\left[\left|\sum_{i,j\in I_{1}}G_{is}^{(1)}(z_{1})G_{tj}^{(1)}(z_{1})G_{ij}^{(1)}(z_{2})\right|^{2}\right]\\
&=&\f{2}{mb_{n}}\sum_{(s,t)\in I_{n}^{+}}\E\left[\left|\sum_{i\in I_{1}}G_{is}^{(1)}(z_{1})G_{ti}^{(1)}(z_{1},z_{2})\right|^{2}\right]\\
&=&\f{2}{mb_{n}}\sum_{(s,t)\in I_{n}^{+}}\E\left[\left|G_{st}^{(1)}(z_{1},z_{2},z_{1})\right|^{2}\right]\\
&\leq&\f{2}{mb_{n}}\E\left[\sum_{s,t=1}^{n}\left|G_{st}^{(1)}(z_{1},z_{2},z_{1})\right|^{2}\right]\\
&=&\f{2}{mb_{n}}\E\left[\|A\|_{Fb}^{2}\right]\\
&=&\f{2}{mb_{n}}\E\left[\sum_{i=1}^{n}\beta_{i}^{2}\right]\\
&\leq&\f{C(z_{1},z_{2})}{mb_{n}}O(b_{n})=O(1),
\Eea
where $\|\cdot\|_{Fb}$ is the Frobenius norm, $\beta_{i}$ are the eigenvalues of $VV^{*}$, and $V$ is the following matrix
\Bea
V_{n\times n}&=&\left[\begin{array}{cccc}
G_{11}^{(1)}(z_{1}) & G_{12}^{(1)}(z_{1}) & \cdots & G_{1k_{n}}^{(1)}(z_{1})\\
G_{21}^{(1)}(z_{1}) & G_{22}^{(1)}(z_{1}) & \cdots & G_{2k_{n}}^{(1)}(z_{1})\\
\vdots & \vdots & &\vdots\\
G_{n1}^{(1)}(z_{1}) &  G_{n2}^{(1)}(z_{1}) & \cdots & G_{nk_{n}}^{(1)}(z_{1})
\end{array}\right]_{n\times k_{n}}
\left[\begin{array}{cccc}
G_{11}^{(1)}(z_{2}) & G_{12}^{(1)}(z_{2}) & \cdots & G_{1k_{n}}^{(1)}(z_{2})\\
G_{21}^{(1)}(z_{2}) & G_{22}^{(1)}(z_{2}) & \cdots & G_{2k_{n}}^{(1)}(z_{2})\\
\vdots & \vdots & &\vdots\\
G_{k_{n}1}^{(1)}(z_{2}) &  G_{k_{n}2}^{(1)}(z_{2}) & \cdots & G_{k_{n}k_{n}}^{(1)}(z_{2})
\end{array}\right]_{k_{n}\times k_{n}}\\
&&\times\left[\begin{array}{cccc}
G_{11}^{(1)}(z_{1}) & G_{12}^{(1)}(z_{1}) & \cdots & G_{1n}^{(1)}(z_{1})\\
G_{21}^{(1)}(z_{1}) & G_{22}^{(1)}(z_{1}) & \cdots & G_{2n}^{(1)}(z_{1})\\
\vdots & \vdots & &\vdots\\
G_{k_{n}1}^{(1)}(z_{1}) &  G_{t2}^{(1)}(z_{1}) & \cdots & G_{k_{n}n}^{(1)}(z_{1})
\end{array}\right]_{k_{n}\times n}.
\Eea
Here we denoted the elements of  set $I_1$ as $I_{1}=\{1,2,\ldots,k_{n}\}.$ Observe that $k_n=2b_n$.
Rank of $V\leq k_{n}=O(b_{n})$. This implies
\Bea
\sum_{i=1}^{n}\beta_{i}^{2}\leq k_{n}C(z_{1},z_{2})=O(b_{n})C(z_{1},z_{2}).
\Eea
Therefore, Var$(T_{1})=O\left(\f{1}{b_{n}^{2}}\right)$, and hence Var$\left\{b_{n}\E_{1}\left[A^{\circ}(z_{1})A^{\circ}(z_{2})\right]\right\}=O\left(\f{1}{b_{n}}\right)$.


Second part of \textbf{(iii)} follows from the following two facts with the help of Cauchy's theorem.
\Bea
&&b_{n}\E_{1}\left[A^{\circ}(z_{1})B^{\circ}(z_{2})\right]=b_{n}\f{d}{dz_{2}}\E_1\left[A^{\circ}(z_{1})A^{\circ}(z_{2})\right]\\
\text{and}\;\;\;&&\text{Var}\left\{b_{n}\E_{1}\left\{A^{\circ}(z_{1})A^{\circ}(z_{2})\right\}\right\}=O\left(\f{1}{b_{n}}\right).
\Eea
Here we skip the details.
\comment{Let us call $g(z_{2})=\E_1[A^{\circ}(z_{1})A^{\circ}(z_{2})]$. Then using the Cauchy integration formula over a closed contour $C$ around $z_{2}$ in the upper half plane and the dominated convergence theorem we have
\Bea
\text{Var}\left\{b_{n}\E_{1}\left\{A^{\circ}(z_{1})B^{\circ}(z_{2})\right\}\right\}&=&\text{Var}\left\{b_{n}\f{d}{dz_{2}}\E_1\left[A^{\circ}(z_{1})A^{\circ}(z_{2})\right]\right\}\\
&=&\text{Var}\left\{b_{n}\f{d}{dz_{2}}g(z_{2})\right\}\\
&=&\text{Var}\left\{\f{1}{\pi i}\oint_{C}\f{b_{n}g(z)\;dz}{(z-z_{2})^{2}}\right\}\\
&=&\f{1}{2\pi i}\oint_{C}\f{1}{(z-z_{2})^{2}}\text{Var}\{b_{n}g(z)\}\;dz\\
&=&O\left[\f{1}{2\pi i}\oint\f{1}{(z-z_{2})^{2}}\;dz\right]O\left(\f{1}{b_{n}}\right)\\
&=&O\left(\f{1}{b_{n}}\right).
\Eea
}


\noindent\textbf{Proof of (iv):} Using \eqref{eqn: bounds of B/A, 1/EA, and EB by 1/Imz} and \eqref{eqn: trace of difference of resolvents interms of A and B}, and proceeding as the proof of proposition \ref{prop: bound on Var(gamma)}, 
\Bea
\E\left[\left|{\g_{n-1}^{\circ}}(z)-\g_{n}^{\circ}(z)\right|^{4}\right]
&=&\E\left[\left|\left(\text{Tr}G^{(1)}(z)-\E[\text{Tr}G^{(1)}(z)]\right)-\left(\text{Tr}G(z)-\E[\text{Tr}G(z)]\right)\right|^{4}\right]\\
&=&\E\left[\left|\f{1+B(z)}{A(z)}-\E\left[\f{1+B(z)}{A(z)}\right]\right|^{4}\right]\\
&\leq&\f{C}{|\Im z|^{8}}\left[\E\left[\left|A^{\circ}\right|^{4}\right]+\E\left[\left|B^{\circ}\right|^{4}\right]+\E\left[\left|A^{\circ}\right|^{4}\right]\right]=O\left(\f{1}{b_{n}^{2}}\right).
\Eea
The last equality follows from the estimates \eqref{eqn: estimate of A^4} and \eqref{eqn: estimate of B^4}.


Using martingale differences as in the proof of Proposition \ref{prop: bound on Var(gamma)},
\Bea
\E\left[\left|\g_{n}^{\circ}\right|^{4}\right]\leq Cn\sum_{k=1}^{n}\E\left[\left|\g_{n}-\E_{k}[\g_{n}]\right|^{4}\right].
\Eea
Consider for $k=1$, others will be similar.
\Bea
\E\left[\left|\g_{n}-\E_{1}[\g_{n}]\right|^{4}\right]&=&\E\left[\left|\text{Tr}(G-G^{(1)})-\E_{1}\left[\text{Tr}(G-G^{(1)})\right]\right|^{4}\right]\\
&=&\E\left[\left|\f{1+B(z)}{A(z)}-\E_{1}\left[\f{1+B(z)}{A(z)}\right]\right|^{4}\right]\\
&\leq&C_{1}(z)\E[|A_{1}^{\circ}|^{4}]+C_{2}(z)\E[|B_{1}^{\circ}|^{4}]\\
&=&O\left(\f{1}{b_{n}^{2}}\right).
\Eea
The last equality follows from \eqref{eqn: estimate of A_1^4} and \eqref{eqn: estimate of B_1^4}. Hence we have the result.\\

\noindent\textbf{Proof of (v):}
Using resolvent identity,
$$(X_2-zI)^{-1}=(X_1-zI)^{-1}+ (X_1-zI)^{-1}(X_1-X_2)(X_2-zI)^{-1},$$
we have
\begin{equation}\label{eqn: resolvent of G_11}zG_{11}(z)=- 1+\sum_{(1,k)\in I_n} m_{1k}G_{k1},\end{equation}
where $I_{n}$ is defined in \eqref{defn: definition of I_1, I_n} and $m_{ij}$s are defined in \eqref{defn: random band matrix}. Now to analyse the terms $\E[m_{1k}G_{k1}]$, we use the following (see eg. \cite{lytova2009central}): Given $\xi$, a real valued random variable with $p+2$ finite moments, and $\phi$, a function from $\mathbb C\to \mathbb R$ with $p+1$ continuous and bounded derivatives then:
\begin{equation}\label{decoupling formula}\E[\xi\phi(\xi)]=\sum_{a=0}^p \frac{\kappa_{a+1}}{a!} \E\left[\phi^{(a)}(\xi)\right]+\epsilon_{p+1}\end{equation}
where $\kappa_a$ is the $a$-th cumulant of $\xi$, $|\epsilon_{p+1}|\leq C\sup_t|\phi^{(p+1)}(t)|\E[|\xi|^{p+2}]$ and  $C$ depends only on $p$.
Since $f_n(z)=\frac 1n \E [\text{Tr} G(z)]=\E [G_{11}(z)],$ using \eqref{eqn: resolvent of G_11} and \eqref{decoupling formula} we get
\bea\label{eqn: master eqn f_n}
zf_n(z)= -1+\sum_{(1,k)\in I_n}\E [m_{1k}G_{k1}] = -1-\sum_{k\in I_1}\frac{1}{b_n} \E\left[G_{k1}^2 + G_{kk}G_{11}\right]+r_n,
\eea
where $r_n$ contains the third cumulant term corresponding to $p=2$ in \eqref{decoupling formula} for $k\neq 1$, and the error terms due to the truncation of the decoupling formula \eqref{decoupling formula}  at $p=2$ for $k\neq 1$ and at $p=0$ for $k=1$. We write \eqref{eqn: master eqn f_n}
\Bea
zf_n(z)&=& -1-\frac{1}{b_n}\E [G_{11}]\E\left[\sum_{k\in I_1}G_{kk}\right]-\frac{1}{b_n} \text{Cov}\left(G_{11},\sum_{k\in I_1}G_{kk}\right)-\frac{1}{b_n}\E\left[\sum_{k\in I_1}G_{k1}^2\right]+r_n\\
&=& -1-f_n(z)( 2f_n(z))-\frac{1}{b_n} \text{Cov}\left(G_{11},\sum_{k\in I_1}G_{kk}\right)-\frac{1}{b_n}\E\left[\sum_{k\in I_1}G_{k1}^2\right]+r_n.
\Eea
 Now, by the Cauchy-Schwarz inequality and \eqref{final result var sum G_{ii}} we get
\Bea
\frac{1}{b_n}\left|\text{Cov}\left(G_{11},\sum_{k\in I_1}G_{kk}\right)\right| &\leq & \frac{1}{b_n} \sqrt{\text{Var} (G_{11})}\sqrt{\text{Var}\left(\sum_{k\in I_1}G_{kk}\right)}\\
&\leq & \frac{1}{b_n}2|\Im z|^{-1}\sqrt{O(|\Im z|^{-4})}\\
&=&O\left(\f{1}{b_{n}|\Im z|^{3}}\right).
\Eea
Also notice that
$$\frac{1}{b_n}\left|\E\left(\sum_{k\in I_1}G_{k1}^2\right)\right|\leq \frac{1}{b_n}|\Im z|^{-2}.$$
We claim  $\displaystyle{r_n=O\left(\frac{1}{b_n|\Im z|^{4}}\right)}$. To prove this, observe that the third cumulant term gives
\begin{equation}\label{eqn: 3rd cumulant}
\frac{\kappa_3}{2b_n^{3/2}}\E\left[\sum_{k\in I_1}2(G_{1k})^3+6G_{11}G_{1k}G_{kk}\right]
\end{equation}
Since $$\sum_{k\in I_1}|G_{1k}|^2\leq \|G\|^{2}\leq |\Im z|^{-2} \ \ \mbox{and}\ \ |G_{ij}|\leq |\Im z|^{-1},$$
we conclude that the third cumulant term contributes $\displaystyle{O\left(\frac{1}{b_n|\Im z|^3}\right)}$ to $r_n$. In a similar manner, the error due to truncation of decoupling formula \eqref{decoupling formula} at $p=2$ is $\displaystyle{O\left(\frac{1}{b_n|\Im z|^4}\right)}$. Similarly, the error term due to truncation of decoupling formula at $p=0$ for $k=1$ is $\displaystyle{O\left(\frac{1}{b_n|\Im z|^2}\right)}$. Thus the claim is proved.
Hence
$$zf_n(z)=-1-2f_n^2(z)+O\left(\frac{1}{b_n|\Im z|^{4}}\right)\ \ \ \mbox{for}\ z\in \mathbb C\bs\mathbb R.$$
Now following similar argument given in the proof of (3.1) in \cite{pizzo2012fluctuations}, one can show that
$$|f_n(z)-f(z)|\leq O\left(\frac{1}{b_n|\Im z|^6}\right)$$
where $f(z)=\frac{1}{4}(-z+\sqrt{z^2-8})$.\\\\
\noindent
\textbf{Proof of (vi):} Recall $A(z)=z-b_n^{-1/2}w_{11}+b_n^{-1}\sum_{i,j\in I_1}G^{(1)}_{ij}w_{1i}w_{1j}$. Now using \eqref{eqn: asymptotic of f_n(z)} with $G$ replaced by $G^{(1)}$, we have
\Bea
(\E[A(z)])^{-1}=\frac{1}{z+b_n^{-1}\sum_{j\in I_1}\E [G^{(1)}_{jj}]}=\frac{1}{z+2f_n(z)}=(z+2f(z))^{-1}+O(b_n^{-1})=-f(z)+O(b_n^{-1})\Eea
Hence $(\E[A(z)])^{-1}=- f(z)+O(b_n^{-1})$. To prove the second part, observe that
\Bea\E [B(z)]=\frac{1}{b_n}\E \left[\sum_{i,j\in I_1}(G^{(1)}G^{(1)})_{ij}w_{1i}w_{1j}\right]=\frac{1}{b_n}\E\left[\sum_{i\in I_1}(G^{(1)}G^{(1)})_{ii}\right]=\frac{1}{b_n}\E\left[\sum_{i\in I_1}\sum_{k=2}^n G_{ik}^{(1)} G_{ki}^{(1)}\right]=\frac{1}{b_n}\sum_{i\in I_1}\frac{d}{dz}G^{(1)}_{ii}
\Eea
Again using \eqref{eqn: asymptotic of f_n(z)} and Cauchy's integral formula, we have
$$\E [B(z)]=\frac{d}{dz}(2f_n(z))=2f'(z)+O(b_n^{-1}).$$
This completes the proof of Lemma \ref{main lemma with poincare condition}.   \end{proof}

\subsection{Proof of \eqref{eqn: limit of ET_n}:}\label{proof of ET_n}
\begin{proof}
We have to find the limit of
$$\E [T_n]=\frac{2}{b_n}\E\left[\sum_{i,j\in I_1}G_{ij}^{(1)}(z)G_{ij}^{(1)}(z_\mu)\right]$$
as $n\to\infty$, where $I_1=\{2\leq i\leq n:(1,i)\in I_n\}$. Let $f,g\in C_b(\mathbb R)$. Define a bilinear form on $C_b(\mathbb R)$ as
\begin{equation}\label{define inner product}
\langle f,g \rangle_n=\frac{1}{b_n}\E \left[\sum_{i,j\in I_1}f(M)_{ij}g(M)_{ji}\right].\end{equation}
Then $\E [T_n]=\langle h(M),h_{\mu}(M)\rangle_{n}$, where $h(x)=(x-z)^{-1}$ and $h_{\mu}(x)=(x-z_{\mu})^{-1}$.
\begin{lem}
For $f,g\in C_b(\mathbb R)$ the limit
$\displaystyle{\langle f,g \rangle=\lim_{n\rightarrow \infty}\langle f,g \rangle_n}$
exists.
\end{lem}
\begin{proof}
The idea of the proof is similar to the proof of Lemma 3.11 of \cite{li2013central}.
First we prove this result for monomials. Although monomials are unbounded, still \eqref{define inner product} makes sense for all $n$, since all moments of the entries of $M$ are  finite. Consider $f(x)=x^l$ and $g(x)=x^m$ where $l,m \in \mathbb N$. Then
\Bea
\langle x^l,x^m \rangle_n &=& \frac{1}{b_n^{1+(l+m)/2}}\mathop{\sum_{(i_0,i_1),(i_1,i_2),\ldots,(i_{l+m-1},i_0)\in I_n}}_{i_o,i_l\in I_1}\E\left[ w_{i_0i_1}w_{i_1i_2}\ldots w_{i_{l+m-1}i_0}\right]\\
\Eea
If $(l+m)$ is odd then $\langle x^l,x^m\rangle_n \rightarrow 0$ using independence of matrix entries and $\E(w_{ij})=0$, and order counting of independent vertices. The argument is similar to the combinatorial  argument given
in the proof of  Wigner semicircular law (see \cite{anderson2010introduction}).  We leave it for the reader.

Now we assume $l+m$ is even. Then
\bea\label{eqn: x^l,x^m interms of (1,i_l)}
\langle x^l,x^m \rangle_n &=& \frac{1}{b_n^{1+(l+m)/2}}\mathop{\sum_{(i_0,i_1),(i_1,i_2),\ldots,(i_{l+m-1},i_0)\in I_n}}_{i_o,i_l\in I_1}\E\left[ w_{i_0i_1}w_{i_1i_2}\ldots w_{i_{l+m-1}i_0}\right]\nonumber \\
&=& \frac{1}{b_n^{1+(l+m)/2}}\mathop{\sum_{(i_0,i_1),(i_1,i_2),\ldots,(i_{l+m-1},i_0)\in I_n}}_{i_o,i_l\in I_1} \E\left[ w_{1i_0}w_{i_0i_1}w_{i_1i_2}\ldots w_{i_{l+m-1}i_0}w_{i_01}\right]+O(b_n^{-1})\nonumber \\
&=&\frac{1}{b_n^{1+(l+m)/2}}\mathop{\sum_{(i_0,i_1),(i_1,i_2),\ldots,(i_{l+m-1},i_0)\in I_n}}_{(1,i_o),(1,i_l)\in I_n} \E\left[ w_{1i_0}w_{i_0i_1}w_{i_1i_2}\ldots w_{i_{l+m-1}i_0}w_{i_01}\right]+O(b_n^{-1})
\eea
The second last equality in \eqref{eqn: x^l,x^m interms of (1,i_l)} holds due to order calculation of independent vertices and independence of matrix entries. Now  define for $k=1,2,\ldots,l+m$,
$$x_k=\left\{\begin{array}{ccc}
i_k-i_{k-1}&\text{if}&|i_k-i_{k-1}|\leq b_n\\
(i_k-i_{k-1})-n&\text{if}&i_k-i_{k-1}>b_n\\
n+(i_k-i_{k-1})&\text{if}&i_k-i_{k-1}<-b_n
\end{array}
\right.
\ \text{with}\ i_{l+m}=i_0, \ \text{and}
$$
$$
x_{0}=\left\{\begin{array}{ccc}
i_0-1&\text{if}&|i_0-1|\leq b_n\\
(i_0-1)-n&\text{if}&i_0-1>b_n\\
\end{array}
\right.
\ \mbox{and} \ \
x_{l+m+1}=\left\{\begin{array}{ccc}
1-i_0&\text{if}&|1-i_0|\leq b_n\\
n+(1-i_0)&\text{if}&1-i_0<-b_n.
\end{array}
\right.
$$
Note, $x_{0}=-x_{l+m+1}$. Since $l,m$ are fixed and $b_n\to \infty$, for large $n$ the restrictions $\{(i_0,i_1),(i_1,i_2),\ldots, (i_{l+m-1},i_0)\in I_n \ \text{and}\ (1,i_0),(1,i_l)\in I_n\}$ are equivalent to $\{|x_0|, |x_1|,\ldots,|x_{l+m}|\leq b_n,\  x_0+x_1+\cdots+x_{l+m}+x_{l+m+1}=0\ \text{and } |x_0+x_1+\cdots+x_l|\leq b_n\}$. Also observe that $x_0+x_1+\cdots+x_{l+m}+x_{l+m+1}=0$ is same as $x_1+\cdots+x_{l+m}=0$  since $x_{0}=-x_{l+m+1}$. Therefore for large $n$
$$\langle x^l,x^m \rangle_n=\frac{1}{b_n^{1+(l+m)/2}}\mathop{\sum_{x_1+\cdots+x_{l+m}=0}}_{|x_i|\leq b_n,0\leq i\leq l+m,\ |x_0+x_1+\cdots +x_l|\leq b_n}\E\left[ w_{1i_0}w_{i_0i_1}w_{i_1i_2}\ldots w_{i_{l+m-1}i_0}w_{i_01}\right)+O(b_n^{-1}].$$

Without loss of generality, we assume that $l\leq m$. Each $\{i_0,i_1,i_2,\ldots,i_{l+m-1},i_0\}$ is a closed path such that distance between the end points of each edge is bounded by $b_n$. As in the proof of Wigner semicircular law only the paths whose edges are pair matched contributes to the limit, here also, only such paths contribute to the limit.  And contribution of each path is $\E(w_{1i_0}w_{i_0i_1}\ldots w_{i_{l+m-1}i_0}w_{i_0 1})=1$ since $\E(w_{ij}^2)=1$. Each such path corresponds to a Dyck path of length $(l+m)$. Recall that a Dyck path $(S(0),S(1),\ldots,S(l+m))$ of length $(l+m)$ satisfies (see \cite{anderson2010introduction})
$$S(0)=S(l+m)=0,\ S(1),S(2),\ldots,S(l+m-1)\geq 0\ \mbox{and}\ |S(i+1)-S(i)|=1,\ \mbox{for}\ i=0,1,\ldots,l+m-1.$$
Specifically, $S(t+1)-S(t)=1$ if the non-oriented edge $(i_t,i_{t+1})$ appears in $\{i_0,i_1,\ldots,i_{l+m-1},i_0\}$ for the first time and $ S(t+1)-S(t)=-1$ if the edge $(i_t,i_{t+1})$ appears in $\{i_0,i_1,\ldots,i_{l+m-1},i_0\}$ for the second time.

Here each Dyck path does not give equal contribution to the limit  due to the condition that $(1,i_l)\in I_n$ and in terms of $x_i$, which is same as $|x_0+x_1+\cdots +x_l|\leq b_n$. We have to take into account this condition. Suppose $S(l)=k$, $0\leq k\leq l$.  Then during the first $l$ steps of the path $\{i_0,i_1,\ldots,i_{l+m-1},i_0\}$,
$k$ edges appear only once and $(l-k)/2$ edges appear twice. The edges appearing twice, the corresponding two number $x_i$ have same absolute value but with different sign. We rename  the remaining $k$ numbers $x_i$  which appear only once as $y_1,y_2,\ldots,y_k$ (according to their order of appearance) and $x_0$ as $y_0$. So the condition $|x_0+x_1+\ldots+x_l|\leq b_n$ reduces to $|y_0+y_1+\ldots+y_k|\leq b_n$. Therefore
\Bea
  \langle x^l,x^m \rangle_n&=&\frac{1}{b_n^{1+(l+m)/2}}\sum_{k=0}^l \#\{\text{Dyck path of length }l+m \text{ with }S(l)=k\}\\
  && \times \#\{|y_0|\leq b_n,|y_1|\leq b_n,\ldots,|y_k|\leq b_n,\ldots,|y_{l+m}|\leq b_n,|y_0+y_1+\cdots +y_k|\leq b_n\}+O(b_n^{-1}).
  \Eea
 and
  \Bea
  \langle x^l,x^m \rangle&=& \lim_{n\to \infty} \langle x^l,x^m \rangle_n\\
  &=&(\sqrt 2)^{l+m+2}\sum_{k=0}^l \#\{\text{Dyck path of length }l+m \text{ with }S(l)=k\}\\
  &&\times \text{Vol}\{|t_0|\leq 1/2,|t_1|\leq 1/2,\ldots,|t_{\frac{l+m}{2}}|\leq 1/2, |t_0+t_1+\cdots +t_k|\leq 1/2\}\\
 &=&(\sqrt 2)^{l+m+2}\sum_{k=0}^l \#\{\text{Dyck path of length }l+m \text{ with }S(l)=k\} \times P(|T_0+T_1+\cdots +T_k|\leq 1/2)\\
\Eea
where $T_0,T_1,\ldots,T_{\frac{l+m}{2}}$ are independent random variables uniformly distributed on $[-1/2,1/2]$. Let $S_{k+1}=T_0+T_1+\ldots +T_k$. Then
$$\E\left[e^{ixS_{k+1}}\right]=\left(\E [e^{ixT_0}]\right)^{k+1}=\left(\frac{\sin x/2}{x/2}\right)^{k+1}.$$
Using inversion formula, the  density of $S_{k+1}$ is given by
$$f_{k+1}(s)=\frac{1}{2\pi}\int_{-\infty}^{\infty}e^{-ixs}\left(\frac{\sin x/2}{x/2}\right)^{k+1}dx.$$
Now $$\gamma_{k+1}:=P(|S_{k+1}|\leq 1/2)=\int_{-1/2}^{1/2}f_{k+1}(s)ds=\frac{1}{2\pi} \int_{-\infty}^{\infty}\left(\frac{\sin x/2}{x/2}\right)^{k+2}dx=f_{k+2}(0),$$
using \cite{cramer1999mathematical} we get exact formula of $\gamma_{k+1}$:
\bea\label{eqn: gamma_{k+1}}
\gamma_{k+1}=\left\{\begin{array}{ccc}
\frac{1}{(k+1)!}\sum_{s=0}^{(k+1)/2}(-1)^s {{k+2}\choose s} \left(\frac{k+1}{2} -s+\frac12\right)^{k+1}&\text{if}& k+1 \text{ even}\\
\frac{1}{(k+1)!}\sum_{s=0}^{k/2}(-1)^s {{k+2}\choose s} \left(\frac{k+1}{2} -s+\frac12\right)^{k+1}&\text{if}& k+1 \text{ odd}.\\
\end{array}
\right.
\eea
The number of Dyck path of length $l+m$ with $S(l)=k$ is
\bea\label{eqn: no. of dyck path}\left[{l \choose \frac{l-k}{2}} -{l \choose \frac{l-k-2}{2}}\right]\times \left[ {m \choose \frac{m-k}{2}}-{m \choose \frac{m-k-2}{2}}\right]=\frac{(k+1)^2}{(l+1)(m+1)}{l+1 \choose \frac{l+k+2}{2}}{m+1 \choose \frac{m+k+2}{2}}.\eea
Hence from \eqref{eqn: gamma_{k+1}} and \eqref{eqn: no. of dyck path}, we get  $$\langle x^l,x^m\rangle=(\sqrt 2)^{l+m+2}C_{l,m}$$
where $C_{l,m}=0$ if $(l+m)$ is odd and
\Bea C_{l,m}&=&\sum_{k=0}^l \frac{(k+1)^2}{(l+1)(m+1)} {l+1 \choose \frac{l+k+2}{2}}{m+1 \choose \frac{m+k+2}{2}}\gamma_{k+1}\\
&=& \left\{
\begin{array}{ccc}
\sum_{k=0}^{l/2}\frac{(2k+1)^2}{(l+1)(m+1)}{l+1 \choose \frac{l-2k}{2}}{m+1 \choose \frac{m-2k}{2}}\gamma_{2k+1}& \text{if}& l \text{ even}\\
\sum_{k=0}^{(l-1)/2}\frac{(2k+2)^2}{(l+1)(m+1)}{l+1 \choose \frac{l-2k-1}{2}}{m+1 \choose \frac{m-2k-1}{2}}\gamma_{2k+2}& \text{if}& l \text{ odd}
\end{array}
\right.
\Eea
if $(l+m)$ is even and $l\leq m$, otherwise, $C_{l,m}=C_{m,l}$. If $f,g$ are polynomials, $f(x)=\sum_{i=0}^pa_i x^i,\ g(x)=\sum_{i=0}^q b_ix^i$, then by linearity
\bea \label{eqn: bilinear form for polynomial}\langle f,g \rangle=\sum_{i=0}^p\sum_{j=0}^q a_ib_j (\sqrt 2)^{i+j+2}C_{i,j}.\eea
For general bounded continuous functions $f,g$, to show that $\langle f,g\rangle$ exists we have to use the Stone-Weierstrass theorem to approximate $f,g$ by appropriate polynomial and then \eqref{eqn: bilinear form for polynomial}. The argument is similar to the argument given in the proof of Lemma 3.11 of \cite{li2013central}. We skip the details.
\end{proof}
In the next lemma we diagonalize the bilinear form $\langle f,g \rangle$.
\begin{lem}
Let $\{U_n(x)\}_{n\geq 0}$ be the rescaled Chebyshev polynomial of the second kind on $[-2\sqrt 2,2\sqrt 2]$,
$$U_n(x)=\sum_{k=0}^{\lfloor \frac n2 \rfloor}(-1)^k {n-k \choose k}\left(\frac {x}{\sqrt 2}\right)^{n-2k}.$$
Then $\{U_n(x)\}$ are orthogonal with respect to the bilinear form \eqref{eqn: bilinear form for polynomial}, that is,
\bea\label{eqn: inner product U_n,U_m}
\langle U_n,U_m \rangle=2\delta_{nm}\gamma_{n+1},
\eea
where $\gamma_{n+1}$ is defined in \eqref{eqn: gamma_{k+1}}.
\end{lem}
\begin{proof}
The proof of this lemma is similar to the proof of Lemma 3.12 of \cite{li2013central}. For sake of completeness we outline it here. Since $\langle x^l,x^m \rangle=0$ if $l+m$ is odd, from linearity $\langle U_l, U_m\rangle =0$ if $l+m$ is odd. We are left to compute $\langle U_{2n},U_{2m} \rangle$ and $\langle U_{2n+1},U_{2m+1}\rangle$. We first compute $\langle x^{2l},U_{2n}\rangle$ and $\langle x^{2l+1},U_{2n+1}\rangle$ for $l=0,1,\ldots,n$.
\Bea
\langle x^{2l},U_{2n}\rangle &=& (\sqrt 2)^{2l+2}\sum_{k=0}^n (-1)^k {2n-k \choose k} C_{2l,2n-2k}\\
&=& (\sqrt 2)^{2l+2}\left [\sum_{k=0}^{n-l}(-1)^k {2n-k \choose k}\sum_{t=0}^l \frac{(2t+1)^2}{(2l+1)(2n-2k+1)}{2l+1 \choose l-t}{2n-2k+1 \choose n-k-t}\gamma_{2t+1}\right.\\
&&\left. + \sum_{k=n-l+1}^{n}(-1)^k {2n-k \choose k}\sum_{t=0}^{n-k} \frac{(2t+1)^2}{(2l+1)(2n-2k+1)}{2l+1 \choose l-t}{2n-2k+1 \choose n-k-t}\gamma_{2t+1}\right]
\Eea
\Bea
&=&(\sqrt 2)^{2l+2} \sum_{t=0}^l \frac{(2t+1)^2}{2l+1} {2l+1 \choose l-t}\left[\sum_{k=0}^{n-t}\frac{(-1)^k(2n-k)!}{k!(n-k-t)! (n-k+t+1)!}\right] \gamma_{2t+1}\\
&=&(\sqrt 2)^{2l+2} \sum_{t=0}^l \frac{(2t+1)^2}{2l+1} {2l+1 \choose l-t} G_1(n,t)\gamma_{2t+1},
\Eea
where $$G_1(n,t)=\sum_{k=0}^{n-t}\frac{(-1)^k(2n-k)!}{k!(n-k-t)! (n-k+t+1)!}.$$
Similarly,
\Bea
\langle x^{2l+1}, U_{2n+1}\rangle &=& (\sqrt2)^{2l+3} \sum_{t=0}^l \frac{(2t+2)^2}{2l+2}{2l+2 \choose l-t}\left[\sum_{k=0}^{n-t}\frac{(-1)^k(2n+1-k)!}{k!(n-k-t)! (n-k+t+2)!}\right] \gamma_{2t+2}\\
&=&(\sqrt2)^{2l+3} \sum_{t=0}^l \frac{(2t+2)^2}{2l+2}{2l+2 \choose l-t}G_2(n,t) \gamma_{2t+2},
\Eea
where $$G_2(n,t)=\sum_{k=0}^{n-t}\frac{(-1)^k(2n+1-k)!}{k!(n-k-t)! (n-k+t+2)!}.$$
$G_1(n,t)$ and $G_2(n,t)$ can be written in terms of hypergeometric function as follows:
\Bea
G_1(n,t)&=&\frac{(2n)!}{(n-t)!(n+t+1)!}\  {}_2{F}_1\left( \stackrel{-(n-t),-(n+t+1)}{-2n};\ 1\right)\\
G_2(n,t)&=&\frac{(2n+1)!}{(n-t)!(n+t+2)!}\  {}_2{F}_1\left( \stackrel{-(n-t),-(n+t+2)}{-2n-1};\ 1\right)\\
\Eea
where ${}_2F_1$ is a hypergeometric function.  By the Chu-Vandermonde identity (see \cite{andrews1999special}), we have
\Bea
 {}_2{F}_1\left( \stackrel{-(n-t),-(n+t+1)}{-2n};\ 1\right)&=&\frac{(-n+t+1)_{n-t}}{(-2n)_{n-t}},\\
 {}_2{F}_1\left( \stackrel{-(n-t),-(n+t+2)}{-2n-1};\ 1\right)&=& \frac{(-n+t+1)_{n-t}}{(-2n-1)_{n-t}},
 \Eea
 where $(a)_n=a(a+1)\cdots(a+n-1)$. Since
 $$(-n+t+1)_{n-t}=\left\{\begin{array}{ccc}
 0&\text{if}&t=0,1,\cdots,n-1\\
 1&\text{if}&t=n
 \end{array}
 \right.
 $$
 we have  $G_1(n,t)=0$, $G_2(n,t)=0$ for $t=0,1,\ldots,n-1$ and $G_1(n,n)=1/(2n+1)$, $G_2(n,n)=1/(2n+2)$. Therefore, $\langle x^{2l}, U_{2n}\rangle =0$ for $0\leq l\leq n-1$ and
$$\langle x^{2n},U_{2n}\rangle=(\sqrt 2)^{2n+2}\gamma_{2n+1}.$$
Similarly, $\langle x^{2l+1}, U_{2n+1}\rangle =0$ for $0\leq l\leq n-1$ and
$$\langle x^{2n+1},U_{2n+1}\rangle=(\sqrt 2)^{2n+3}\gamma_{2n+2}.$$
Therefore
$$\langle U_{2n},U_{2n}\rangle=2 \gamma_{2n+1} \ \mbox{and}\ \langle U_{2n+1},U_{2n+1} \rangle=2 \gamma_{2n+2}.$$
This completes the proof of the lemma.\end{proof}
Now we  complete the proof of  \eqref{eqn: limit of ET_n}. For $f,g \in C_b(\mathbb R)$, if
$$f_k=\frac{1}{4\pi}\int_{-2\sqrt2}^{2\sqrt2}f(x)U_k(x)\sqrt{8-x^2}dx,\ \ g_k=\frac{1}{4\pi}\int_{-2\sqrt2}^{2\sqrt2}g(x)U_k(x)\sqrt{8-x^2}dx,$$
then
\bea\label{eqn: f,g in terms of f_k,g_k}
\langle f,g \rangle &=& \sum_{k=0}^{\infty}f_k g_k 2  \gamma_{k+1}\\
&=& \frac{1}{8\pi^3}\int_{-2\sqrt 2}^{2\sqrt 2} \int_{-2\sqrt 2}^{2\sqrt 2} f(x)g(y)\sqrt{8-x^2}\sqrt{8-y^2}\left[\pi\sum_{k=0}^{\infty}U_k(x)U_k(y)\gamma_{k+1}\right]dxdy\nonumber \\
&=&\frac{1}{8\pi^3}\int_{-2\sqrt 2}^{2\sqrt 2} \int_{-2\sqrt 2}^{2\sqrt 2} f(x)g(y)\sqrt{8-x^2}\sqrt{8-y^2}F(x,y)dxdy\nonumber
\eea
where
\bea\label{eqn: tildeF(x,y)}
F(x,y)=\pi\sum_{k=0}^{\infty}U_k(x)U_k(y)\gamma_{k+1}=2 \int_{-\infty}^{\infty}\frac{z-z^3}{2(1-z^2)^2+z^2(x^2+y^2)-z(1+z^2)xy}ds
\eea
with $z=\frac{\sin s}{s}$. Now \eqref{eqn: f,g in terms of f_k,g_k} holds due to \eqref{eqn: inner product U_n,U_m} and orthogonality of Chebyshev polynomial  with respect to the Wigner semicircular law, that is,
$$\int_{-2\sqrt2}^{2\sqrt 2}U_n(x)U_m(x)\frac{1}{4\pi}\sqrt{8-x^2}dx=\delta_{mn}.$$
And \eqref{eqn: tildeF(x,y)} is a straightforward consequence of the Fourier analysis using the following fact
$$U_n(x)=\frac{\sin[(n+1)\theta]}{\sin \theta},\ x=2\sqrt 2 \cos \theta.$$
This completes the proof of Proof of \eqref{eqn: limit of ET_n}.
\end{proof}

\section*{Recent development} Recently, after submission of our paper, M. Shcherbina \cite{shcherbina2015fluctuations} improved our result by removing the restriction $b_{n}>>\sqrt{n}$ and proved it for all $b_{n}$ which satisfies $b_{n}\to \infty$ and $\f{b_{n}}{n}\to 0$ as $n\to\infty$.
\section*{Acknowledgment} We thank the referee for her/his constructive comments which have led to a significant improvement in presentation.

\section{Some MATLAB simulation results}\label{section: Simulation}

Here is what we found in MATLAB simulations.

\begin{figure}[h!]
\begin{subfigure}[b]{0.5\textwidth}
\includegraphics[scale=0.3]{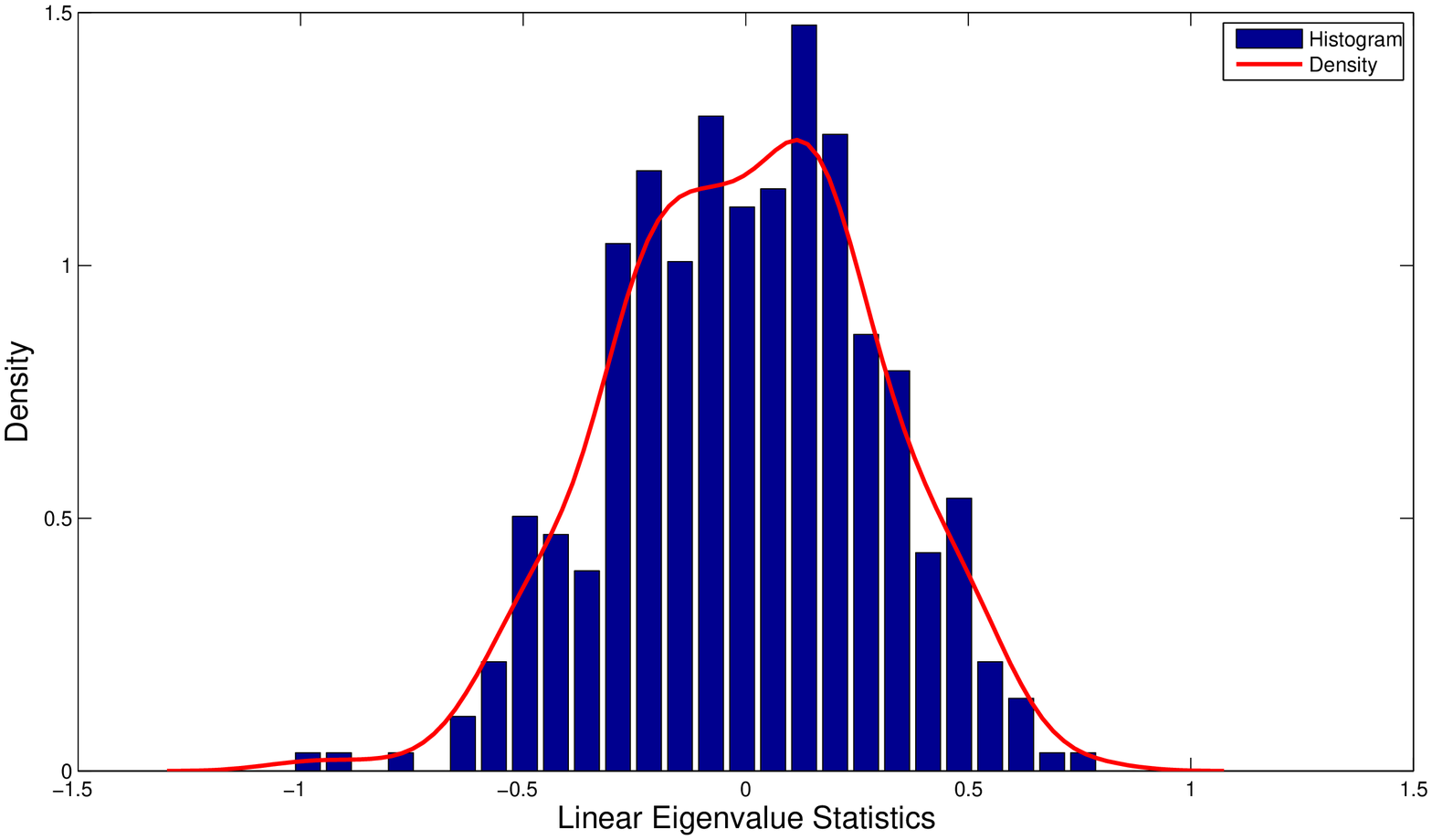}
\caption{$n=2000$, $b_{n}=n^{0.2}$. Fourth moment/(variance)$^2$=2.92}
\end{subfigure}
\hfill
\begin{subfigure}[b]{0.5\textwidth}
\includegraphics[scale=0.3]{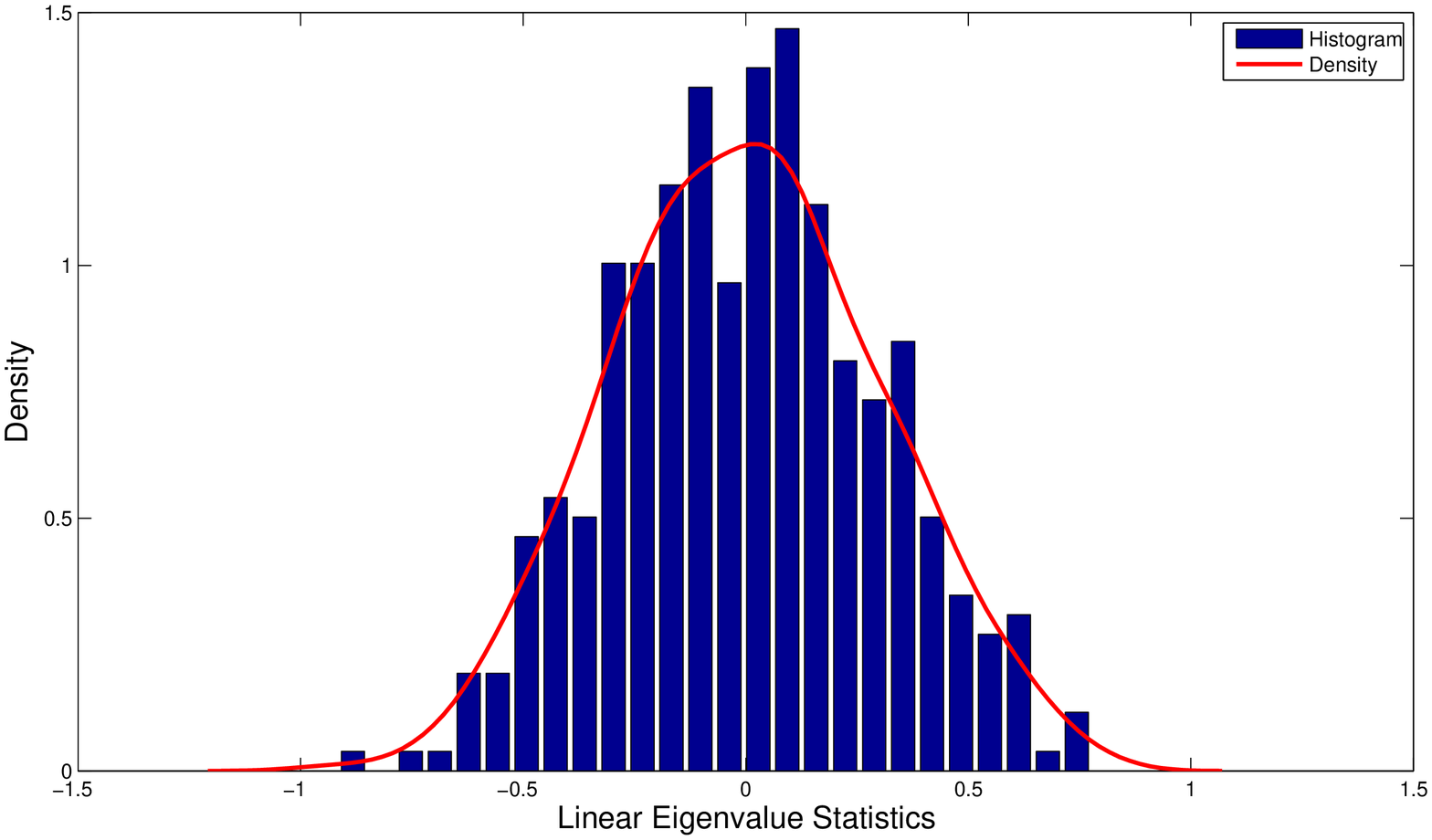}
\caption{$n=2000$, $b_{n}=n^{0.4}$. Fourth moment/(variance)$^2$=2.71}
\end{subfigure}
\end{figure}
\begin{figure}[h!]
\begin{subfigure}[b]{0.5\textwidth}
\includegraphics[scale=0.3]{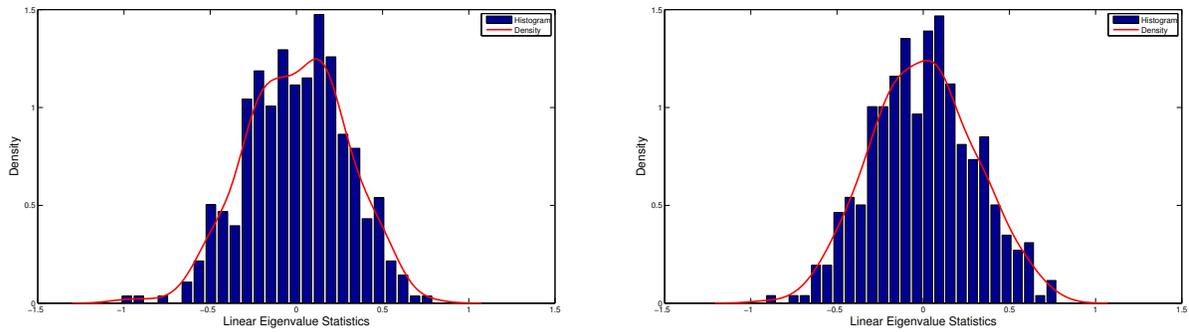}
\caption{$n=2000$, $b_{n}=n^{0.6}$. Fourth moment/(variance)$^2$=2.57}
\end{subfigure}
\hfill
\begin{subfigure}[b]{0.5\textwidth}
\includegraphics[scale=0.3]{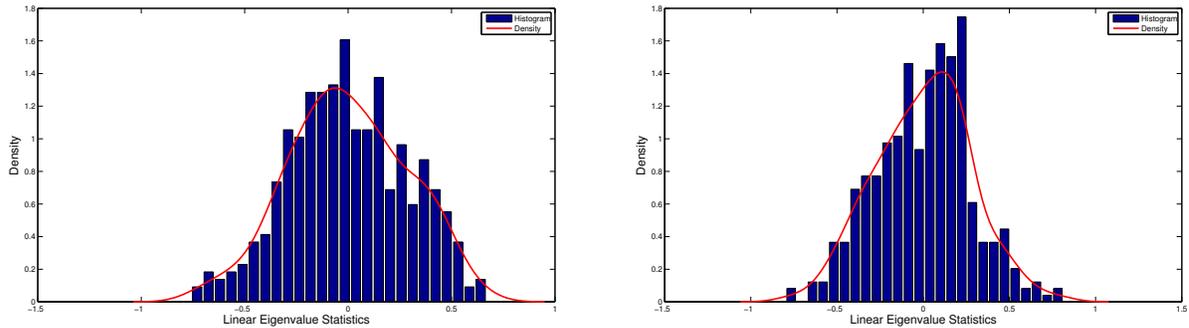}
\caption{$n=2000$, $b_{n}=n^{0.8}$. Fourth moment/(variance)$^2$=2.91}
\end{subfigure}
\caption{The eigenvalue statistics was sampled $400$ times. The test function was $\phi(x)=\sqrt{16-x^{2}}.$}
\end{figure}

In the following example we had taken a different test function.
 
\begin{figure}[h!]
\begin{subfigure}[b]{0.4\textwidth}
\includegraphics[scale=0.3]{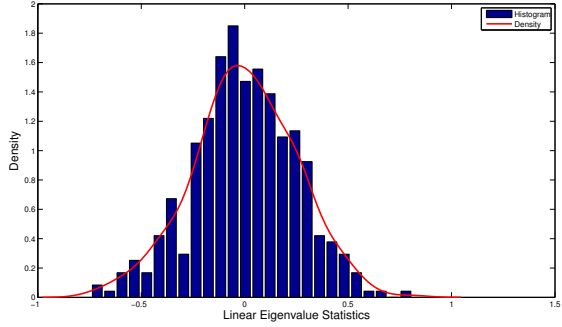}
\caption{$n=2000$, $b_{n}=n^{0.2}$. Fourth moment/(variance)$^2$=3.08}
\end{subfigure}
\hfill
\begin{subfigure}[b]{0.4\textwidth}
\includegraphics[scale=0.3]{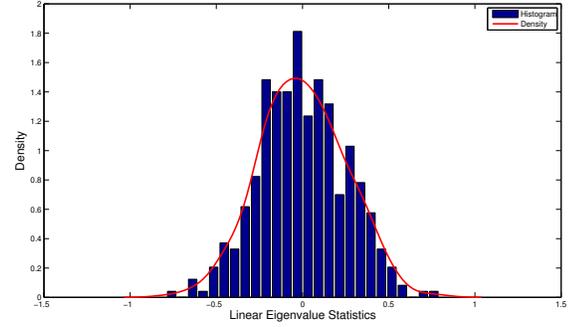}
\caption{$n=2000$, $b_{n}=n^{0.4}$. Fourth moment/(variance)$^2$=2.94}
\end{subfigure}
\begin{subfigure}[b]{0.4\textwidth}
\includegraphics[scale=0.3]{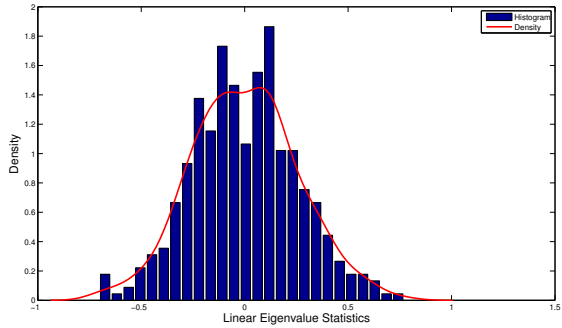}
\caption{$n=2000$, $b_{n}=n^{0.6}$. Fourth moment/(variance)$^2$=3.00}
\end{subfigure}
\hfill
\begin{subfigure}[b]{0.4\textwidth}
\includegraphics[scale=0.3]{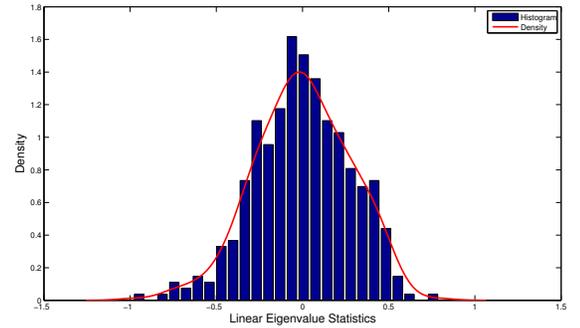}
\caption{$n=2000$, $b_{n}=n^{0.8}$. Fourth moment/(variance)$^2$=3.08}
\end{subfigure}
\caption{The eigenvalue statistics was sampled $400$ times. The test function was $\phi(x)=e^{-x^{2}}.$}
\end{figure}

\newpage
\bibliographystyle{abbrv}
\bibliography{CLT for Band.bbl}
\end{document}